\documentclass{amsart}
\usepackage{amsmath}
\usepackage{amscd}
\usepackage{amssymb}

\newenvironment{pf}{{\medskip \bf\em Proof.}}{{\hfill$\square$}\medskip}

\newcommand{\be}{\beta}

\def\({\left(}
\def\){\right)}
\def\be{\begin{equation}}
\def\ee{\end{equation}}

\def\<{\langle}
\def\>{\rangle}

\newcommand{\dbar}{\overline{\partial}}

\newcommand{\ddt}[1]{\frac{\partial #1}{\partial t}}

\newcommand{\Vol}{\operatorname{Vol}}
\newcommand{\ddbar}{\frac{\sqrt{-1}}{2} \partial\dbar}

\newcommand{\cal}{\mathcal}

\newtheorem{theorem}{Theorem}[section]
\newtheorem{theorem/definition}{Theorem/Definition}[section]
\newtheorem{proposition}{Proposition}[section]
\newtheorem{lemma}{Lemma}[section]
\newtheorem{corollary}{Corollary}[section]
\theoremstyle{remark}
\newtheorem{remark}{Remark}[section]

\theoremstyle{definition}
\newtheorem{definition}{Definition}[section]

\begin{document}
\title
{The K\"ahler-Ricci flow on Fano manifolds}
\author{Huai-Dong Cao}
\address{Department of Mathematics, University of Macau, Macao, China \& Department of Mathematics\\ Lehigh University\\
Bethlehem, PA 18015} \email{huc2@lehigh.edu}

\maketitle
\date{}

\footnotetext[1]{Partially supported by NSF grant DMS-0909581}

\setcounter{tocdepth}{1}
\tableofcontents

\noindent {\bf Introduction}

\medskip
In these lecture notes, we aim at giving an introduction to the K\"ahler-Ricci flow (KRF) on Fano manifolds, i.e.,
compact K\"ahler manifolds with positive first Chern class. It will cover some of the developments of the KRF in its first
twenty years (1984-2003),  especially an essentially self-contained exposition of Perelman's uniform estimates 
on the scalar curvature,  the diameter, and the Ricci potential function (in $C^1$-norm)  for the normalized K\"ahler-Ricci flow (NKRF), 
including the monotonicity of Perelman's $\mu$-entropy and $\kappa$-noncollapsing theorems for the Ricci flow on compact
manifolds.  Except in the last section where we shall briefly discuss the
formation of singularities of the KRF in Fano case, much of the
recent progress since Perelman's uniform estimates are not
touched here, especially those by Phong-Sturm \cite{PS06}
and Phong-Song-Sturm-Weinkove \cite {PSSW09, PSSW08, PSSW11} (see 
also \cite{Pali, CaoMeng09, Sz10, To10, MS, Zhang11} etc.) tying the convergence of the NKRF
to a notion of GIT stability for the diffeomorphism group, in the spirit of the conjecture of Yau \cite{Yau93} (see also
\cite{Ti97, Do02}). We hope to discuss these developments, as well as many works related to
K\"ahler-Ricci solitons, on another occasion. 
We also refer the readers to the recent lecture notes by J. Song and B.
Weinkove \cite{SW} for some of the other significant developments in KRF.

In spring 1982, Yau invited Richard Hamilton to give a talk at the Institute for Advanced Study (IAS)
on his newly completed seminal work
``Three-manifolds with positive Ricci curvature" \cite{Ha82}. Shortly after, Yau asked me, Ben Chow and Ngaiming Mok to present
Hamilton's work on the Ricci flow in details at Yau's IAS geometry seminar. At the time, Ben Chow and I were first year graduate students, and Mok was an instructor
at Princeton University. There was another fellow first year graduate student, S. Bando, working with Yau.  It was clear to us that Yau
was very excited about Hamilton's work and saw its great potential. He encouraged us to study and pursue Hamilton's Ricci flow. 

Besides attending courses at Princeton and Yau's lecture series in geometric analysis at IAS, I spent most of 1982 preparing for Princeton's  General Examination,
a 3-hour oral exam covering two basic subjects (Real \& Complex Analysis and Algebra) plus two additional advanced topics. But I also continued to
study Hamilton's paper.
After I passed the General Exam in January 1983, I went to see Yau and asked for his suggestion for a thesis topic.
Yau immediately gave me the problem to study the Ricci flow on K\"ahler manifolds, especially the long time existence and convergence on Fano manifolds.   
At the time I hardly knew any complex geometry (but I did not dare to tell Yau so). In the following months, I spent a lot of time reading and trying to understand Yau's seminal paper on the Calabi conjecture \cite{Yau78}, and also Calabi's
paper on extremal K\"ahler metrics \cite{Calabi} suggested by Yau.  In the mean time, it happened that Yau invited Calabi to visit IAS in spring 1983 and I benefited a great deal 
from Calabi's lecture series on ``Vanishing theorems in K\"ahler geometry" at IAS that spring.

By spring 1984 I had managed to prove the long time existence of the canonical K\"ahler-Ricci flow by adopting Yau's celebrated a priori estimates for the Calabi conjecture 
to the parabolic case,  as well as the convergence to K\"ahler-Einstein metrics when the first Chern class $c_1$ is either negative or zero. The convergence proof when $c_1=0$ used a 
version of the Li-Yau type estimate for positive solutions to the heat equation with evolving metrics and an argument of J. Moser. But little progress was made 
toward long time behavior when $c_1>0$. Without fully aware of the significance and the difficulties of the problem
at the time, 
I felt kind of uneasy that I did not meet my adviser's high expectation. But to my relief, Yau seemed
quite pleased and encouraged me to write up the work. That resulted my 1985 paper \cite{Cao85}.  In Fall of 1984, several of Yau's Princeton graduate students, including me and B. Chow,
followed him to San Diego where both Richard Hamilton and Rick Schoen also arrived.  By then Bando had used the short time property
of the flow to classify three-dimensional compact K\"ahler manifolds of nonnegative bisectional curvature (see \cite{Bando})  and graduated from Princeton.
Shortly after our arrival in San Diego, following Hamilton's work in \cite{Ha86}, Ben Chow and I  also used the short time property of the flow to classify
compact K\"ahler manifolds with nonnegative curvature operator in all dimensions  \cite {CaoChow}. In 1988,  Mok's work \cite{Mo88} was published in
which he was able to show (in 1986) nonnegative bisectional curvature is preserved in all dimensions. By combining the short time property of the flow and the
existence of special rational curves by Mori \cite{Mori},  Mok proved the generalized
Frankel conjecture in its full generality  (see also a recent new proof by H. Gu \cite{Gu09}).  Around the same time, Tsuji \cite{Tsuji} extended my work on the KRF for the negative Chern class case 
to compact complex manifolds of general type (see also the related later work of Tian-Zhang \cite{TZhang}). This is a brief history of the KRF in its early years.

Late 1980s and 1990s saw great advances in the Ricci flow by Hamilton \cite{Ha88, Ha93, Ha93E, Ha95, Ha95F, Ha97, Ha99}
which laid the foundation to use the Ricic flow to attack the Poincar\'e and geometrization conjectures. In particular, the works of Hamilton \cite{Ha88} and Ben Chow \cite{Chow91}
imply that every metric on a compact Riemann surface can be deformed to a metric of constant curvature under the Ricci flow. 
During the same period, there were several developments in the KRF,  including  
the constructions of $U(n)$-invariant K\"ahler-Ricci soliton examples by Koiso \cite{Ko} and
the author \cite{Cao94}\footnote{My work was carried out at Columbia University in early 1990s.}; the Li-Yau-Hmailton inequalities and the Harnack inequality
for the KRF  \cite{Cao92, Cao97};  the important work of W.-X. Shi \cite{Sh90B, Sh97}, another former student of Yau, using the noncompact KRF
to approach Yau's conjecture that a complete noncompact K\"ahler manifold with positive bisectional curvature is biholomorphic to the complex Euclidean
space $\mathbb C^n$ (see \cite{ChTam08} for a recent survey on the subject), etc.   In addition,  in 1991 at Columbia University,
I first observed that Mabuchi's K-energy \cite{Ma86} and the functional
defined in Ding-Tian \cite{DT92} are monotone decreasing under the KRF \cite{Cao91}. The fact that the K-energy
is monotone under the KRF turned out to be quite useful, and was first applied in the work of Chen-Tian \cite{CTi02} ten years later. 

In November 2002 and spring 2003, Perelman \cite{P1, P2, P3} made astounding breakthroughs in the Ricci flow. In April 2003, in a private lecture at MIT,
Perelman presented in detail his uniform scalar curvature and diameter estimates for the NKRF based on the monotonicity of his $\mathcal W$-functional
and $\mu$-entropy, and the powerful ideas in his $\kappa$-noncollapsing results.  We remark that prior to Perelman's lecture at MIT, such uniform estimates had
appeared only in the important special case when NKRF has positive bisectional curvature, in the work of Chen and Tian \cite{CTi02} for the K\"ahler surface case  (see also \cite{CTi06}
for the higher dimensional case) assuming in addition the existence of K-E metrics; 
and also in the work of B.-L Chen, X.-P. Zhu and the author  \cite{CCZ03} in all dimensions and without assuming the existence
of K-E metrics.

From Hamilton and Perelman's works to the recent proof of the
1/4-pinching differentiable sphere theorem by Brendle-Schoen
\cite{BS09}, we have seen spectacular applications of the Ricci
flow and its sheer power of flowing to canonical metrics/structures without a priori knowing their existence. Let
us hope to see similar phenomena happen to the KRF.

\bigskip
\noindent {\bf Acknowledgements}. \ This article was written in Spring 2012. It is based on a mini-course on KRF delivered at University of Toulouse III  in February 2010,  a talk on Perelman's uniform estimates for NKRF at Columbia University's Geometry and Analysis Seminar  in Fall 2005, and several conference talks, including ``Einstein Manifolds and Beyond" at CIRM
(Marseille - Luminy, fall 2007), ``Program on Extremal K\"ahler Metrics and K\"ahler-Ricci Flow" at the De Giorgi Center
(Pisa, spring 2008), and ``Analytic Aspects of Algebraic and Complex Geometry" at CIRM (Marseille - Luminy, spring 2011). This article served as the Lecture Notes by the author for a 
graduate course at Lehigh University in spring 2012, as well as a short course at the Mathematical Sciences Center of Tsinghua University in May, 2012. 
I would like to thank Philippe Eyssidieux, Vincent Guedj, and Ahmed Zeriahi for inviting me to give the mini-course in Toulouse, and especially Vincent Guedj for
inviting me to write up the notes for a special volume. I also wish to thank the participants in my courses, especially Qiang Chen, Xin Cui, Chenxu He, Xiaofeng Sun, Yingying Zhang and Meng Zhu, for 
their helpful suggestions. Finally, I would like to take this opportunity to express my deep gratitude to Professors E. Calabi, R. Hamilton, and
S.-T. Yau for teaching me the K\"ahler geometry, the Ricci flow, and geometric analysis over the years.

\section{Preliminaries}

In this section, we fix our notations and recall some basic facts and formulas in K\"ahler Geometry.

\medskip
\noindent {\bf 1.1 K\"ahler metrics and K\"ahler forms}
\medskip

 Let $(X^n, g)$ be a compact K\"ahler manifold of complex dimension $n$ with the K\"ahler metric $g$.  In local holomorphic coordinates $(z^1,  \cdots, z^n)$, denote its K\"ahler form by
$$ \omega= \frac{\sqrt{-1}}{2}\sum_{i,j} g_{i\bar j} dz^i\wedge d\bar{z}^j. \eqno(1.1)$$ By definition, $g$ is K\"ahler means that its K\"ahler
 form $ \omega$ is a {\it closed} real (1,1) form, or equivalently,
 $$\partial_k g_{i\bar j}= \partial_i g_{k\bar j} \quad{\mbox{and}} \quad \partial_{\bar k} g_{i\bar j}= \partial_{\bar j}g_{i\bar k} \eqno(1.2)$$
 for all $i,j,k=1, \cdots n$.  Here $\partial_k=\partial/\partial z^k$ and $\partial_{\bar k}=\partial/\partial {\bar z}^{k}$.

 The cohomology class [$\omega$] represented by $\omega$ in $H^2(X, \mathbb R)$ is called the K\"ahler class of the metric $g_{i\bar j}$. By the Hodge theory, two K\"ahler metrics $g_{i\bar j}$ and $\tilde g_{i\bar j}$ belong to the same K\"ahler class  if and only if  $g_{i\bar j}=\tilde g_{i\bar j}+\partial_i\partial_{\bar j}\varphi$, or equivalently,
 $$ \omega=\tilde\omega + \ddbar \varphi \eqno(1.3)$$ for some real valued smooth function $\varphi$ on $X$.

 The volume of $(X, g)$ is given by
 $$\Vol (X, g)=\int_X \omega^{[n]}, \eqno(1.4)$$
 where we have followed the convention of Calabi \cite{Calabi} to denote  $\omega^{[n]} = \omega^{n}/n!$ so that the volume form is given by
$$ dV=\det(g_{i\bar j}) \wedge_{i=1}^{n}(\frac{\sqrt{-1}}{2} dz^{i}\wedge d\bar{z}^{i}) =\omega^{[n]}.  \eqno(1.5)$$
Clearly, by Stokes' theorem, if $g$ and $\tilde g$ are in the same K\"ahler class then we have
 $$\Vol (X, g)=\Vol (X, \tilde g).$$

\noindent {\bf 1.2 Curvatures and the first Chern class}

\medskip

The Christoffel symbols of the metric  $g_{i\bar j}$ are given by
$$\Gamma_{ij}^k=g^{k\bar \ell}\partial_i g_{j\bar \ell} \quad \mbox{and} \quad \Gamma_{\bar i \bar j}^{\bar k}=g^{\ell \bar k }\partial_{\bar i} g_{\ell \bar j}, \eqno(1.6)$$
where $(g^{i\bar j})= ((g_{i\bar j})^{-1})^{T}$.  It is a basic fact in K\"ahler geometry that, for each point $x_0\in X^n$, there exists a system of holomorphic normal coordinates $(z^1, \cdots, z^n)$
at $x_0$ such that
 $$ g_{i\bar j}(x_0)=\delta_{i\bar j} \qquad \mbox{and} \qquad \partial_k g_{i\bar j}(x_0)=0, \qquad \forall i,j, k=1, \cdots n. \eqno(1.7) $$

The curvature tensor of the metric $g_{i\bar j}$ is defined as
$R^{\ j}_{i \  k\bar\ell}=- \partial_{\bar \ell} \Gamma_{ik}^j$,
or by lowering $j$ to the second index,
$$R_{i\bar j k\bar \ell}=g_{p\bar j} R^{\ p}_{i \ k\bar\ell}=-\partial_k\partial_{\bar \ell} g_{i\bar j} +g^{p\bar q}\partial_k g_{i\bar q}\partial_{\bar\ell}g_{p\bar j}. \eqno(1.8)$$
From (1.2) and (1.8), we immediately see that $R_{i\bar j k\bar
\ell}$ is symmetric in $i$ and $k$, in $\bar j$ and $\bar \ell$,
and in the pairs $\{i\bar j\}$ and $\{k\bar\ell\}$.

We say that $(X^n, g)$ has positive (holomorphic) bisectional curvature, or  positive holomorphic sectional curvature, if
$$ R_{i\bar j k\bar \ell} v^i v^{\bar j} w^k w^{\bar\ell}>0, \quad \mbox{or } \quad R_{i\bar j k\bar \ell} v^i v^{\bar j} v^k v^{\bar\ell}>0$$
respectively, for all nonzero vectors $v$ and $w$ in the holomorphic tangent bundle $T_x X$ of $X$ at $x$ for all $x\in X$.

The Ricci tensor of the metric $g_{i\bar j}$ is obtained by taking the trace of $R_{i\bar j k\bar \ell}$:
$$R_{i\bar j}= g^{k\bar\ell} R_{i\bar j k\bar \ell}=-\partial_i\partial_{\bar j}\log\det (g). \eqno(1.9)$$
From (1.9), it is clear that the Ricci form
$${Ric} =\frac{\sqrt{-1}}{2}\sum_{i,j} R_{i\bar j} dz^i\wedge d\bar{z}^j \eqno(1.10)$$
is real and closed. It is well known that the first Chern class $c_1 (X)\in  H^2(X, \mathbb{Z})$ of $X$ is represented by the Ricci form:
$${[Ric]}=\pi c_1 (X). \eqno(1.11)$$

Finally, the scalar curvature of the metric $g_{i\bar j}$ is
$$R=g^{i\bar j} R_{i\bar j}. \eqno(1.12)$$ Hence, the total scalar curvature
$$\int_X R dV=  \int _X Ric\wedge \omega^{[n-1]}, \eqno(1.13)$$ depends only on the K\"ahler class of $\omega$ and the first Chern class  $c_1(X)$. \\

\noindent {\bf 1.3 Covariant derivatives}

\medskip

Given any smooth function $f$, we denote by
$$\nabla_i f=\partial_i f, \qquad \nabla_{\bar i} f=\partial_{\bar i} f.$$
For any (1,0)-form $v_i$, its covariant derivatives are defined as
$$\nabla_j  v_i=\partial_j v_i-\Gamma^{k}_{ij} v_k \quad \mbox{and} \quad \nabla_{\bar j}  v_i=\partial_{\bar j} v_i.  \eqno(1.14)$$
Similarly, for covariant 2-tensors, we have
$$ \nabla_k v_{i\bar j}=\partial_k v_{i\bar j}-\Gamma^{p}_{ik}v_{p\bar j},  \quad  \nabla_{\bar k} v_{i\bar j}=\partial_{\bar k} v_{i\bar j}-\Gamma^{\bar p}_{\bar j\bar k}v_{i\bar p},$$
$$ \nabla_k v_{ij}=\partial_k v_{ij}-\Gamma^{p}_{ik}v_{pj} -\Gamma^{p}_{jk}v_{ip}, \quad \mbox{and} \quad \nabla_{\bar k} \ v_{ij}=\partial_{\bar k} v_{ij}.$$
Now, in the K\"ahler case, the second Bianchi identity in Riemannian geometry translates into the relations
$$ \nabla_p R_{i\bar j k\bar \ell} = \nabla_k R_{i\bar j p\bar \ell} \quad \mbox{and} \quad
\nabla_{\bar p} R_{i\bar j k\bar \ell} = \nabla_{\bar \ell} R_{i\bar j k\bar p}. \eqno(1.15)$$

Covariant differentiations of the same type can be commuted freely, e.g.,
$$  \nabla_k\nabla_{j}  v_i = \nabla_j\nabla_{k}  v_i,  \qquad  \nabla_{\bar k}\nabla_{\bar j}  v_i = \nabla_{\bar j}\nabla_{\bar k}  v_i, \eqno(1.16)$$ etc.
But we shall need the following formulas when commuting covariant derivatives of different types: 
$$\nabla_k\nabla_{\bar j}  v_i - \nabla_{\bar j}\nabla_{k}  v_i= - R_{k\bar j i\bar \ell}  v_{\ell}, \eqno(1.17)$$
$$ \nabla_k\nabla_{\bar \ell} v_{i\bar j} - \nabla_{\bar \ell}\nabla_{ k} v_{i\bar j}= - R_{k\bar \ell i\bar p }  v_{p \bar j} + R_{k\bar \ell p\bar j} v_{i \bar p}, \eqno(1.18)$$
etc.

We define
$$|\nabla f|^2=g^{i\bar j} \partial_i f \partial_{\bar j} f, \eqno(1.19)$$
$$|Rc|^2= g^{i\bar\ell} g^{k\bar j} R_{i\bar j} R_{k\bar \ell}, \eqno(1.20)$$
and
$$|Rm|^2= g^{i\bar q} g^{p\bar j} g^{k\bar s} g^{r\bar \ell} R_{i\bar j k\bar \ell} R_{p\bar q r \bar s}. \eqno(1.21)$$
The norm square  $|S|^2$ of any other type of covariant tensor $S$ is defined similarly.

Finally, the Laplace operator on a tensor $S$ is, in normal coordinates,  defined as
$$\Delta S=\frac 1 2 \sum_k (\nabla_k\nabla_{\bar k} + \nabla_{\bar k}\nabla_{k}) S. \eqno(1.22)$$

\medskip
\noindent {\bf 1.4  K\"ahler-Einstein metrics and K\"ahler-Ricci solitons}

\smallskip

 It is well known that a  K\"ahler metric $g_{i\bar j}$ is K\"ahler-Einsten if
$$R_{i\bar j}=\lambda g_{i\bar j} $$ for some real number $\lambda \in \mathbb R$.  K\"ahler-Ricci solitons are extensions of K-E metrics: a  K\"ahler metric
$g_{i\bar j}$ is called a gradient K\"ahler-Ricci (K-R) soliton if  there exists a real-valued smooth function $f$ on $X$ such that
$$R_{i\bar j}=\lambda g_{i\bar j} -\partial_i\partial_{\bar j} f    \quad\mbox{and}\quad    \nabla_i\nabla_{j} f=0. \eqno(1.23)$$
It is called {\sl shrinking} if $\lambda>0$, {\sl steady} if $\lambda=0$, and {\sl expanding} if $\lambda<0$.  The function $f$ is called a potential function.

Note that the second equation in (1.23) is equivalent to saying the gradient vector field
$$\nabla f=(g^{i\bar j}\partial_{\bar j}f) \frac{\partial} {\partial z^{i}}$$ is holomorphic.  By scaling,
we can normalize $\lambda=1, 0, -1$ in (1.23).  The concept of Ricci soliton was introduced by Hamilton \cite{Ha88} in mid 1980s. It has since played a
significant role in Hamilton's Ricci  flow as Ricci solitons often arise as singularity models (see, e.g., \cite{Cao10} for a survey).
Note that when $f$ is a constant function, K-R solitons are simply
K-E metrics.

Clearly, if $X^n$ admits a K-E metric  or K-R soliton $g$ then the first Chern class is necessarily definite, as
$$\pi c_1(X)= \lambda  [\omega_g]. $$ When $c_1(X)=0$ it follows from Yau's solution to the Calabi conjecture
that in each K\"ahler class there exists a unique Calabi-Yau metric  (i.e., Ricci-flat K\"ahler metric) $g$ in that class. Moreover, when $c_1(X)<0$,  Aubin \cite{Aubin}
and Yau \cite{Yau78} proved  independently that there exists a unique K\"ahler-Einstein metric in the class $-\pi c_1(X)$.

However,  in the Fano case (i.e., $c_1(X)>0$), it is well known that there exist obstructions to the existence of a K-E metric $g$ in the class of $\omega\in \pi c_1(X)$
with $R_{i\bar j}=g_{i\bar j}$.  One of the obstructions is the Futaki invariant defined as follows:
take any K\"ahler metric $g$ with $\omega \in \pi c_1(X)$. Then its K\"ahler class [$\omega$] agrees with
its Ricci class [$Ric$]. Hence, by the Hodge theory, there exists a real-valued smooth function $f$, called the Ricci potential of the metric $g$, such that
$$ R_{i\bar j}=g_{i\bar j}-\partial_i\partial_{\bar j} f.  \eqno(1.24)$$
In \cite{Fu}, Futaki proved that the functional $F\!:  \eta (X) \rightarrow C$ defined by
$$ F(V)=\int_X \nabla_{\!V}\!f \  \!\omega^{[n]}=\int_X (V\!\cdot\nabla \!f )\  \!\omega^{[n]}  \eqno(1.25)$$ on the space $\eta (X)$ of holomorphic vector fields depends only on the class $\pi c_1(X)$, but not the metric $g$.
In particular, if a Fano manifold $X^n$ admits a positive K-E metric, then the Futaki invariant $F$ defined above must be zero.

On the other hand, it turns out that compact stead and expanding K-R soliotns are necessarily K-E. If $g$ is a {\it shrinking} K-R soliton satisfying
$$R_{i\bar j}=g_{i\bar j} - \partial_i\partial_{\bar j} f    \quad\mbox{and}\quad    \nabla_i\nabla_{j} f=0 \eqno(1.26)$$ with non-constant function $f$ then, taking
$V=\nabla f$,  we have
$$ F(\nabla f)=\int_X |\nabla f|^2  \omega^{[n]} \ne 0. \eqno(1.27)$$
The existence of compact shrinking K-R solitons were first shown independently by Koiso \cite{Ko} and the author \cite{Cao94},
and later by X. Wang and X. Zhu \cite{WZ}. Dancer-Wang \cite{DW11} extended my construction to the general case when the base manifold is a product of Fano K-E manifolds.
The noncompact example shrinking K-R soliton was first found by Feldman-Ilmanen-Knopf \cite{FIK}, see also Dancer-Wang \cite{DW11} and Futaki-Wang
\cite{FW11} for further examples. 

We remark that Bando and Mabuchi \cite{BM87} proved that positive K-E metrics are unique in the sense that any two positive K-E metrics on $X^n$
only differ by an automorphism of $X^n$. Moreover, Tian and Zhu \cite{TZ02} extended the definition of the Futaki invariant by introducing a corresponding
obstruction to the existence of shrinking K-R solitons on Fano manifolds. They also proved the Bando-Mabuchi type uniqueness
result for shrinking K-R solitons \cite{TZ00}.

\bigskip

\section {The K\"ahler-Ricci flow and the normalized K\"ahler-Ricci flow}

In this section we introduce  the K\"ahler-Ricci flow (KRF) and the normalized K\"ahler-Ricci flow (NKRF) on Fano manifolds, i.e.,
compact K\"ahler manifolds with positive first Chern class.  We state the basic long time existence of  solutions to the NKRF
proved by the author in \cite{Cao85}, derive the evolution equations of various curvature tensors, and present Mok's result on
preserving the non-negativity of the holomorphic bisectional curvature under the KRF.

\medskip
\noindent {\bf 2.1 The K\"ahler-Ricci flow and the normalized K\"ahler-Ricci flow}

\medskip
 On any given K\"ahler manifold $(X^n, \tilde g_{i\bar j})$, 
the K\"ahler-Ricci flow deforms the initial metric  $\tilde g$ by the equation
$$\frac{\partial}{\partial t} g_{i\bar j}(t) = - R_{i\bar j}(t), \quad g(0) = \tilde g, \eqno(2.1)$$
or equivalently, in terms of the K\"ahler forms, by
$$\frac{\partial}{\partial t} \omega (t) = - \textrm{Ric}(\omega (t)), \quad \omega(0) = \omega_0. \eqno(2.1')$$

Note that, by (2.1'), the K\"ahler class $[\omega(t)]$ of the evolving metric $g_{i\bar j}(t)$ satisfies the ODE
$$\frac {d} {dt} [\omega(t)]=-\pi c_1(X),$$ from which it follows that 
$$ [\omega(t)]= [\omega_0] - t\pi c_1(X). \eqno(2.2)$$ 

\begin{proposition} Given any initial K\"ahler metric $\tilde g$ on a compact K\"ahler manifold $X^n$, KRF (2.1) 
admits a unique solution $g(t)$ for a short time. 
\end{proposition} 

\begin{proof}  We consider the nonlinear, strictly parabolic, scalar equation of Monge-Amper\'e type  
$$ \ddt{\varphi} = \log \frac{\det(\tilde g_{i\bar j}-t \tilde R_{i\bar j}+\partial_i\partial_{\bar j}\varphi)} {\det(\tilde g_{i\bar j})}, \quad \varphi (0)=0$$ as in  \cite{Bando}. 
Then, this parabolic equation admits a unique solution $\varphi$ for a short time, and it is easy to verify that 
$$ g_{i\bar j}(t)=: \tilde g_{i\bar j}-t \tilde R_{i\bar j}+\partial_i\partial_{\bar j}\varphi$$
gives rise to a short time solution to KRF (2.1) for small $t>0$. This proves the existence. For the uniqueness, suppose $h_{i\bar j}$ is another solution to KRF (2.1). Then, by (2.2), we  have
$$ h_{i\bar j}= \tilde g_{i\bar j}-t \tilde R_{i\bar j}+\partial_i\partial_{\bar j}\psi$$ for some real-valued function $\psi$. So it follows that 
 $$\partial_i\partial_{\bar j} (\ddt{\psi})=-R_{i\bar j} +\tilde R_{i\bar j}. $$ Hence,  by (1.9) and by adjusting with an additive function in $t$ only, we have
$$ \ddt{\psi} = \log \frac{\det(\tilde g_{i\bar j}-t \tilde R_{i\bar j}+\partial_i\partial_{\bar j}\psi)} {\det(\tilde g_{i\bar j})}.$$
Note also that $h_{i\bar j}(0)=\tilde g_{i\bar j}$ forces $\psi(0)$ to be a constant function.
Therefore $\varphi$ and $\psi$ differ by a function in $t$ only which in turn implies that $g=h$.

Alternatively, by the work of Hamilton \cite{Ha82} (see also De Turck \cite{deTurck}), there exists a unique solution $g(t)$ to (2.1), 
regarded as the Ricci flow for Riemannian metric, for a short time with $\tilde g$ as the initial metric. Moreover, Hamilton \cite{Ha95F} 
observed that the holonomy group does not change under the Ricci flow for a short time. Thus, the solution $g(t)$ remains
K\"ahler for $t>0$.
\end{proof}

\begin{lemma} Under the K\"ahler-Ricci flow (2.1), the volume of $(X, g_{i\bar j}(t))$ changes by
$$\frac{d}{dt} \Vol (X, g(t))=- \int_X R(t) \ \omega^{[n]}(t).$$
\end{lemma}

\begin{proof} Under KRF (2.1), we have 
$$\frac{\partial}{\partial t} \omega^{[n]} = (\frac{\partial}{\partial t} \log \det(g_{i\bar j}))  \omega^{[n]} $$ and
$$\frac{\partial}{\partial t} \log \det(g_{i\bar j})=g^{i\bar j} \frac{\partial}{\partial t}g_{i\bar j}=-g^{i\bar j} R_{i\bar j}=-R.$$
Therefore, the volume element $dV=\omega^{[n]}$ changes by $$\frac{\partial}{\partial t} \omega^{[n]}=-R  \omega^{[n]}. \eqno(2.3)$$

\end{proof}

From now on, we consider a Fano manifold $(X^n, \tilde g_{i\bar j})$ such that 
$$[\omega_0]=[\tilde \omega]=\pi c_1(X), \eqno(2. 4) $$ 
and we deform the initial metric $\tilde g$ by the KRF (2.1). 

To keep the volume unchanged, we consider the normalized K\"ahler-Ricci flow

$$\frac{\partial}{\partial t} g_{i\bar j} = - R_{i\bar j} + g_{i\bar j}, \quad g(0) = \tilde g  \eqno(2.5)$$
or equivalently
$$\frac{\partial}{\partial t} \omega = - \textrm{Ric}(\omega) + \omega, \quad \omega(0) = \omega_0. \eqno(2.5')$$

From the proof of Lemma 2.1, it is easy to see that the following holds (in fact, under NKRF (2.5) the solution $g(t)$ has the same K\"ahler class): 

\begin{lemma} Under the normalized K\"ahler-Ricci flow (2.5), we have
$$\frac{\partial}{\partial t}(dV)=(n-R) dV.$$
\end{lemma}

By (2.2) and (2.4), it follows that under the KRF  (2.1)
$$ [\omega(t)]= \pi (1- t) c_1(X),$$ showing that $[\omega(t)]$ shrinks homothetically and would become degenerate at $t=1$. This suggests that if
$[0, T)$ is the maximal existence time interval of solution $\hat g(t) $ to KRF (2.1), then $T$ cannot exceed $1$. We shall see that the NKRF (2.5)
has solution $g(t)$ exists for all time $0\le t< \infty$, which in turn implies that $T=1$ for KRF (2.1).

By direct calculations, one can easily verify the following relations between the solutions to KRF (2.1) and NKRF (2.5).

\begin{lemma} Let $\hat g_{i\bar j}(s), 0\le s<1,$ and $g_{i\bar j} (t), 0\le t<\infty, $ be solutions to the {\rm KRF} (2.1) and the {\rm NKRF} (2.5) respectively. Then,
$\hat g_{i\bar j}(s)$ and $g_{i\bar j}(t)$ are related by
$$ \hat g_{i\bar j}(s)=  (1-s) g_{i\bar j}(t(s)), \quad  t=-\log (1-s)$$   and
$$g_{i\bar j}(t)=  e^{t}  \hat g_{i\bar j}(s(t)), \quad s= 1- e^{-t}. $$
\end{lemma}
\medskip

\begin{corollary} Let $\hat g_{i\bar j}(s)$ and $g_{i\bar j} (t)$ be as in Lemma 2.3. Then, their scalar curvatures
and the norm square of their curvature tensors are related respectively by
$$ (1-s) \hat R(s)= R(t(s)),$$
and
$$ (1-s) |\hat{Rm}|_{\hat g(s)}= |Rm|_{g(t(s))}.$$
\end{corollary}

\bigskip

\noindent {\bf 2.2 The long time existence of the NKRF}

\medskip

First of all, it is well known that the NKRF (2.5) is equivalent to a parabolic scalar equation of complex Monge-Amp\`ere type on the K\"ahler potential.
For any given initial metric  $g_{0}=\tilde g$ satisfying (2.4), consider
$$g_{i\bar j} (t) =\tilde g_{i\bar j}+\partial_i\partial_{\bar j} \varphi, \eqno(2.6)$$
where $\varphi=\varphi(t)$ is a time-dependent, real-valued, smooth unknown function on $X$.
Then,
$$\frac{\partial}{\partial t} g_{i\bar j}= \partial_i \partial_{\bar j} \varphi_{t}$$
and
\begin{align*}
-R_{i\bar j}+g_{i\bar j} & = -R_{i\bar j}+ \tilde g_{i\bar j}+\partial_i\partial_{\bar j}  \varphi
= -R_{i\bar j}+ \tilde R_{i\bar j}+\partial_i\partial_{\bar j} ( \tilde f+\varphi)\\
& = \partial_i\partial_{\bar j} \log \frac{\omega^n} {{\tilde \omega}^n} + \partial_i\partial_{\bar j} ( \tilde f+\varphi).
\end{align*}
Here $\tilde f$ is the Ricci potential of $\tilde g_{i\bar j}$ as defined in (1.24).  Thus, the NKRF (2.5) reduces to
$$ \partial_i \partial_{\bar j} \varphi_{t} =
\partial_i\partial_{\bar j} \log \frac{\omega^n} {{\tilde \omega}^n} + \partial_i\partial_{\bar j} ( \tilde f+\varphi), $$
or equivalently,
$$ \ddt{\varphi} = \log \frac{\det(\tilde g_{i\bar j}+\partial_i\partial_{\bar j}\varphi)} {\det(\tilde g_{i\bar j})}+\tilde f +\varphi +b(t) \eqno(2.7)$$
for some function $b(t)$ of $t$ only.

Note that (2.7) is strictly parabolic, so standard PDE theory implies its short time existence  (cf. \cite{Baker}). Clearly, we have

\begin{lemma} If $\varphi $ solves the parabolic scalar equation  (2.7), then
$g_{i\bar j} (t)$, as defined in (2.6), is a solution to the {\rm NKRF (2.5)}.
\end{lemma}

Now we can state the following long time existence result shown by the author \cite{Cao85}, based on the parabolic version of Yau's a priori estimates in \cite{Yau78}.
We refer the readers to \cite{Cao85}, or the lecture notes by Song and Weinkove \cite{SW} in this volume, for a proof.

\begin{theorem} [{Cao \cite{Cao85}}] The solution $\varphi(t)$ to {\rm(2.7)} exists for all time $0\le t <\infty$. Consequently, the solution $g_{i\bar j}(t)$ to the normalized K\"ahler-Ricci
flow {\rm (2.5)} exists for all time  $0\le t <\infty$.
\end{theorem}

\medskip

\noindent {\bf 2.3 Preserving positivity of the bisectional curvature}

\medskip
To derive the curvature evolution equations for both KRF and NKRF, we consider
$$\frac{\partial}{\partial t} g_{i\bar j} = - R_{i\bar j} + \lambda g_{i\bar j}. \eqno(2.8)$$

\begin{lemma} Under (2.8), we have
$$ \frac{\partial} {\partial t} R_{i\bar j} =\Delta R_{i\bar j} +R_{i\bar j k\bar \ell} R_{\ell\bar k} -R_{i\bar k}R_{k\bar j}, \eqno (2.9 )$$ and
$$ \frac{\partial} {\partial t} R=\Delta R +|Rc|^2-\lambda R. \eqno (2. 10)$$
\end{lemma}

\begin{pf} First of all, from (1.9), we get
$$\frac{\partial} {\partial t} R_{i\bar j}= - \nabla_i\nabla_{\bar j} (g^{k\bar \ell} \frac{\partial} {\partial t} g_{k\bar \ell}) = \nabla_i\nabla_{\bar j} R. \eqno (2.11)$$
On the other hand, by using the commuting formulas (1.16)-(1.18) for covariant differentiations, we have
\begin{align*}
 \nabla_k\nabla_{\bar k} R_{i\bar j}  =  \nabla_k\nabla_{\bar j} R_{i\bar k} & =  \nabla_{\bar j} \nabla_k R_{i\bar k} -R_{k\bar j i\bar \ell} R_{\ell \bar k} + R_{k\bar j \ell \bar k} R_{i \bar \ell} \\
& = \nabla_{\bar j} \nabla_i R -R_{i\bar j k\bar \ell} R_{\ell \bar k} + R_{i \bar \ell} R_{l\bar j} ,
\end{align*}
and
$$\nabla_k\nabla_{\bar k} R_{i\bar j} =  \nabla_{\bar k}\nabla_{k} R_{i\bar j}.$$
Hence,
$$\Delta R_{i\bar j} =\frac{1}{2} \( \nabla_k\nabla_{\bar k} + \nabla_{\bar k}\nabla_{k}\) R_{i\bar j} = \nabla_i \nabla_{\bar j} R -R_{i\bar j k\bar \ell} R_{\ell \bar k}
+ R_{i \bar \ell} R_{l\bar j}. \eqno (2. 12)$$
Therefore, (2.9) follows from (2.11) and (2.12)

Next, using the evolution equation of $R_{i\bar j}$, we have
\begin{align*}
\frac{\partial} {\partial t} R  = \frac{\partial} {\partial t} (g^{i\bar j}R_{i\bar j})
& = g^{i\bar j} (\Delta R_{i\bar j} +R_{i\bar j k\bar \ell} R_{\ell\bar k}-R_{i\bar k}R_{k\bar j}) + R_{i\bar j} (R_{j\bar i}-\lambda g_{j\bar i})\\
& = \Delta R +|Rc|^2 -\lambda R.
\end{align*}
\end{pf}

\begin{lemma} Under (2.8), we have
\begin{align*}
\frac{\partial} {\partial t} R_{i\bar j k\bar \ell} = & \Delta R_{i\bar j k\bar \ell} +R_{i\bar j p\bar q} R_{q\bar p k\bar \ell} +R_{i\bar \ell p\bar q} R_{q\bar p k\bar j}
- R_{i\bar p k\bar q} R_{p\bar j q\bar \ell} + \lambda R_{i\bar j k\bar \ell} \\
& -\frac {1} {2} (R_{i\bar p} R_{p\bar j k\bar \ell} +R_{p\bar j} R_{i\bar p k\bar \ell} +R_{k\bar p} R_{i\bar j p\bar \ell} + R_{p\bar \ell}R_{i\bar j k\bar p}  ).
\end{align*}
\end{lemma}

\begin{pf}  From (1.8) and by using normal coordinates, we have
\begin{align*}
 \frac{\partial} {\partial t} R_{i\bar j k\bar \ell}  & =  \partial_k\partial_{\bar \ell} R_{i\bar j} + \lambda R_{i\bar j k\bar \ell}
 =  \partial_k (\nabla_{\bar \ell} R_{i\bar j} +\Gamma_{\bar j\bar \ell}^{\bar p} R_{i\bar p}) + \lambda R_{i\bar j k\bar \ell} \\
 & =  \nabla_k\nabla_{\bar \ell} R_{i\bar j} -  R_{i\bar p}R_{p\bar j k\bar \ell} + \lambda R_{i\bar j k\bar \ell}.
 \end{align*}

On the other hand, by (1.15)  and covariant differentiation commuting formulas (1.16)-(1.18), we obtain
\begin{align*}
\nabla_{\bar p} \nabla_p R_{i\bar j k \bar \ell} = \nabla_k \nabla_{\bar l} R_{i\bar j} -R_{i\bar j p\bar q} R_{q\bar p k\bar \ell} + R_{i\bar p k\bar q} R_{p\bar j q\bar \ell}-R_{i\bar \ell p\bar q} R_{q\bar p k\bar j}
+R_{i\bar j p\bar \ell} R_{k\bar p},
\end{align*}
and
$$\nabla_p\nabla_{\bar p} R_{i\bar j k \bar \ell} =\nabla_{\bar p}\nabla_p R_{i\bar j k \bar \ell} - R_{i\bar q} R_{q\bar j k\bar \ell}  +R_{q\bar j} R_{i\bar q k\bar \ell} -R_{k\bar q} R_{i\bar j q\bar \ell}
+R_{q\bar \ell}  R_{i\bar j k\bar q}.$$
Hence,
 \begin{align*}
 \Delta R_{i\bar j k \bar \ell}  & =\frac{1}{2} \(\nabla_p\nabla_{\bar p} + \nabla_{\bar p}\nabla_p\) R_{i\bar j k \bar \ell} \\
& =  \nabla_k \nabla_{\bar l} R_{i\bar j} -R_{i\bar j p\bar q} R_{q\bar p k\bar \ell} + R_{i\bar p k\bar q} R_{p\bar j q\bar \ell}-R_{i\bar \ell p\bar q} R_{q\bar p k\bar j} \\
& \ \ \  +\frac{1}{2} (- R_{i\bar p} R_{p\bar j k\bar \ell}  +R_{p\bar j} R_{i\bar p k\bar \ell} +R_{k\bar p} R_{i\bar j p\bar \ell}
+R_{p\bar \ell}  R_{i\bar j k\bar p}),
 \end{align*}
 and Lemma 2.6 follows.
\end{pf}

\begin{remark} Clearly, the Ricci evolution equation (2.9) is also a consequence of Lemma 2.6, but the proof in Lemma 2.5 is more direct and easier.
\end{remark}

The Ricci flow in general seems to prefer positive curvatures: positive Ricci curvature is preserved in three-dimension \cite{Ha82}; positive scalar curvature,
positive curvature operator \cite{Ha86} and positive isotropic curvature \cite{BS09, Ng} are preserved in all dimensions. Here we present a proof of Mok's
theorem that positive bisectional curvature is preserved under KRF.

\begin{theorem}  [{Mok \cite{Mo88}}] Let $(X^n, \tilde g)$ be a compact K\"ahler manifold of nonnegative
holomorphic bisectional curvature, and let $g_{i\bar j}(t)$ be a solution to the KRF (2.1) or NKRF (2.5) on $X^n\times [0, T)$.
Then, for $t>0$, $g_{i\bar j} (t)$ also has nonnegative holomorphic bisectional curvature.
Furthermore, if the holomorphic bisectional curvature is positive at one point at $t=0$, then $g_{i\bar j}(t)$ has positive holomorphic
bisectional curvature at all points for $t>0$.
\end{theorem}

\begin{pf} Let us denote by
\begin{align*}
F_{i\bar j k\bar \ell}=: & R_{i\bar j p\bar q} R_{q\bar p k\bar \ell} - R_{i\bar p k\bar q} R_{p\bar j q\bar \ell} +R_{i\bar \ell p\bar q} R_{q\bar p k\bar j}  +\lambda R_{i\bar j k\bar \ell} \\
& -\frac {1} {2} (R_{i\bar p} R_{p\bar j k\bar \ell} +R_{p\bar j} R_{i\bar p k\bar \ell} +R_{k\bar p} R_{i\bar j p\bar \ell} + R_{p\bar \ell}R_{i\bar j k\bar p}  )
\end{align*}
so that
$$\frac{\partial} {\partial t} R_{i\bar j k\bar \ell} = \Delta R_{i\bar j k\bar \ell} + F_{i\bar j k\bar \ell}.$$
By a version of Hamilton's strong tensor maximum principle (cf. \cite{Bando}), it suffices to show that the following ``null-vector condition"  holds: for any (1,0) vectors
$V=(v^i)$ and $W=(w^i)$, we have
$$F_{i\bar j k\bar \ell} v^iv^{\bar j}w^kw^{\bar \ell}\ge 0 \quad{\mbox{whenever}} \quad R_{i\bar j k\bar \ell} v^iv^{\bar j}w^kw^{\bar \ell}=0, \eqno ({\rm NVC})$$
or simply,
$$F_{V\bar V W\bar W}=: \!F(V, \bar V, W, \bar W) \ge 0 \quad{\mbox{whenever}} \quad R_{V\bar V  W\bar W}=: \! {\rm Rm} (V, \bar V, W, \bar W)=0.$$

\medskip
{\bf Claim 2.1}:  If $ R_{V\bar V  W\bar W}=0$,  then for any $(1,0)$ vector $Z$, we have
$$ R_{V\bar Z  W\bar W}=R_{V\bar V  W\bar Z}=0.$$

{\sl Proof}. For real parameter $s\in \mathbb R$, consider
$$G(s) = {\rm Rm} (V+sZ, \bar V +s \bar Z, W, \bar W).$$ Since the bisectional curvature is nonnegative and $R_{V\bar V  W\bar W}=0$, it follows that $G'(0)=0 $ which implies
that
$$ {\rm Re} \ (R_{V\bar Z  W\bar W})=0.$$
Suppose $R_{V\bar Z  W\bar W} \ne 0$, and let  $R_{V\bar Z  W\bar W}= |R_{V\bar Z  W\bar W}| e^{\sqrt{-1}\theta}$. Then, replacing $Z$ by $e^{-\sqrt{-1}\theta}Z$ in the above,
we get $$0= {\rm Re} \ (e^{-\sqrt{-1}\theta} R_{V\bar Z  W\bar W})=|R_{V\bar Z  W\bar W}|,$$ a contradiction. Thus, we must have $$R_{V\bar Z  W\bar W} = 0.$$
Similarly, we have $R_{V\bar V  W\bar Z}=0$.

\medskip

By Claim 2.1, we see that if $R_{V\bar V  W\bar W}=0$ then
$$F_{V\bar V W\bar W} = R_{V\bar V Y \bar Z} R_{Z \bar Y W \bar W} - |R_{V \bar Y W \bar Z}|^2  + |R_{V \bar W Y \bar Z}|^2.$$
Therefore, (NVC) follows immediately from the following
\medskip

{\bf Claim 2.2}:  Suppose $R_{V\bar V  W\bar W}=0$. Then,  for any $(1,0)$ vectors $Y$ and $Z$,
$$R_{V\bar V  Y\bar Z} R_{Z\bar Y  W\bar W} \ge |R_{V\bar Y  W\bar Z}|^2.$$

{\sl Proof}.  Consider
\begin{align*}
H(s)  & = {\rm Rm} (V+sY, \bar V +s \bar Y, W+sZ, \bar W +s \bar Z) \\
& = s^2 \(R_{V\bar V Z\bar Z} + R_{Y\bar Y W\bar W} +R_{V\bar Y W\bar Z} +R_{Y\bar V Z\bar W}+R_{V\bar Y Z\bar W}+R_{Y\bar V W\bar Z}\) +O (s^3).
\end{align*}
Here we have used Claim 2.1.

Since  $H(s)\ge 0$ and $H(0)=0$,  we have $H''(0)\ge 0.$ Hence, by taking $Y=\zeta^{k}e_{k}$ and $Z= \eta^{\ell} e_{\ell}$ with respect to any basis $\{e_1, \cdots e_n\}$, 
we obtain a real, semi-positive definite bilinear form $Q(Y,Z)$:
 \begin{align*}
0\le  Q (Y, Z) =: & \! R_{V\bar V Z\bar Z} + R_{Y\bar Y W\bar W} +R_{V\bar Y W\bar Z} +R_{Y\bar V Z\bar W} +R_{V\bar Y Z\bar W}+R_{Y\bar V W\bar Z}\\\
=& R_{V\bar V e_{k} \bar e_{\bar \ell}} \eta^{k}\eta^{\bar \ell} + R_{e_{k} \bar e_{\bar \ell} \bar W \bar W} \zeta^{k} \zeta^{\bar \ell} +R_{V\bar e_{k} W\bar e_{\ell}}\zeta^{\bar k}\eta^{\bar \ell}
 +R_{e_{k}\bar V e_{\ell}\bar W} \zeta^{k}\eta^{\ell}\\
 & +R_{V \bar e_{k}  e_{\ell}\bar W}\zeta^{\bar k}\eta^{\ell}
 +R_{e_{k} \bar V W \bar e_{\ell}} \zeta^{k}\eta^{\bar \ell}
 \end{align*}

 Next, we need a useful linear algebra fact (cf. Lemma 4.1 in \cite{Cao92}):

 \begin{lemma} Let $A$ and $C$ be two $m \times m$ real symmetric semi-positive definite matrices, and let $B$ be a real $m \times m$ matrix
 such that the $2m \times 2m$ real symmetric matrix
$$
G_1=\left(
\begin{array}{cc} A&  B\\
B^{T} &  C
\end{array}
\right)
$$
is semi-positive definite. Then, we have
$$ {\rm Tr}  (A C)\ge |B|^2.$$
 \end{lemma}

 {\sl Proof}.  Consider the associated matrix
 $$
G_2=\left(
\begin{array}{cc} C&  -B\\
-B^{T} &  A
\end{array}
\right).
$$
It is clear that $G_2$ is also symmetric and semi-positive definite. Hence, we get
$${\rm Tr} (G_1G_2)\ge 0.$$ However,
 $$
G_1G_2=\left(
\begin{array}{cc} AC-BB^{T}&  BA-AB\\
B^{T}C-CB^{T} &  CA-B^{T}B
\end{array}
\right).
$$
Therefore, $$ {\rm Tr}(AC)-|B|^2 =\frac{1}{2} {\rm Tr}(G_1G_2) \ge 0.$$

As a special case, by taking 
$$G_1= \left(
\begin{array}{cccc}
ReA & -ImA & Re(B+D)^T &-Im(B+D)^T \\
ImA & ReA & Im(B-D)^T & Re(B-D)^T\\
Re(B+D)& Im(B-D)&ReC&-ImC\\
-Im(B+D)&Re(B-D)&ImC&ReC\\
\end{array}
\right),
$$
we immediately obtain the following (see [Lemma 4.2, Cao92])\\ 

\begin {corollary} {\sl Let $A, B, C, D $ be complex matrices with $A$ and $C$ being Hermitian. Suppose that the (real) quadratic form 
\begin{equation*}
\sum A_{k\bar l}\eta^k\overline{\eta^{l}}+C_{k\bar l}\zeta^k\overline{\zeta^{l}}+2Re(B_{k\bar {l}}\eta^k\overline{\zeta^{l}})+2Re(D_{k l}\eta^k\zeta^{l}),  \quad \eta, \zeta\in {\mathbb C}^n, 
\end{equation*}
is semi-positive definite.  Then we have}
\begin{equation*}
Tr(AC)\ge|B|^2+|D|^2, 
\end{equation*}
i.e., 
\begin{equation*}
\sum A_{k\bar l}C_{l\bar k}\ge \sum |B_{k\bar l}|^2+|D_{kl}|^2.
\end{equation*}
\end{corollary}

Now, by applying Corollary 2.2 to the above semi-positive deÞnite real bi- 
linear form Q, one gets
$$R_{V\bar V  Y\bar Z} R_{Z\bar Y  W\bar W} \ge |R_{V\bar Y  W\bar Z}|^2 + |R_{V\bar W  Y\bar Z}|^2.$$ 

We have thus proved (NVC), which concludes the proof of Theorem 2.2.
\end{pf}

\medskip
\begin{remark} S. Bando \cite{Bando} first proved Theorem 2.2 for $n=3$, and W. -X. Shi \cite{Sh97} extended Theorem 2.2 to the complete noncompact case 
with bounded curvature tensor. 
\end{remark} 

Furthermore, by slightly modifying the above proof of Theorem 5.2.11, R. Hamilton and the author \cite{CH92} observed in 1992 at IAS that nonnegative  {\sl holomorphic orthogonal bisectional curvature},
$Rm(V,\bar V, W, \bar W)\ge 0$ whenever $V\!\perp\!W$, is also preserved under KRF.  For the reader's convenience, we provide the proof below.\\

\begin{theorem} [{Cao-Hamilton}] Let $g_{i\bar j}(t)$ be a solution to the KRF (2.1) on a complete K\"ahler manifold with bounded curvature. If $g_{i\bar j}(0)$ has nonnegative
holomorphic orthogonal bisectional curvature, then it remains so for $g_{i\bar j}(t)$ for $t>0$.
\end{theorem}

\begin{pf} First of all, by using a certain special evolving orthonormal frame $\{e_{\alpha}\}$ under KRF (2.1) similarly as in \cite{Ha86} (see also [Section 5, \cite{Sh97}]), 
one obtains the simplified evolution equation

$$
\label{eq:CurvatureEvolution}
\frac{\partial} {\partial t} R_{\alpha\bar \beta \gamma \bar \delta} = \Delta R_{\alpha\bar \beta \gamma \bar \delta} +R_{\alpha\bar \beta \mu \bar \nu } R_{\nu \bar \mu \gamma \bar \delta} 
+ R_{\alpha\bar \delta \mu \bar \nu } R_{\nu \bar \mu \gamma \bar \beta}  -R_{\alpha\bar \mu \gamma \bar \nu } R_{\mu \bar \beta \nu \bar \delta},  \eqno(2.13)
$$
where $R_{\alpha\bar \beta \gamma \bar \delta}$ is the Riemannian curvature tensor components with respect to the evolving frame $\{e_{\alpha}\}$.

Again, by Hamilton's tensor maximal principle, it suffices to check the corresponding null-vector condition: 
$$
G_{\alpha\bar\alpha\beta\bar\beta}\ge 0, \text{ whenever } R_{\alpha\bar\alpha\beta\bar\beta}=0 \text{ for any } e_\alpha \perp e_\beta, \eqno ({\rm NVC'})
$$
where

$$
G_{\alpha\bar\beta\gamma\bar\delta}=R_{\alpha\bar \beta \mu \bar \nu } R_{\nu \bar \mu \gamma \bar \delta} + R_{\alpha\bar \delta \mu \bar \nu } R_{\nu \bar \mu \gamma \bar \beta}-R_{\alpha\bar \mu \gamma \bar \nu } R_{\mu \bar \beta \nu \bar \delta}. $$

Now, without loss of generality,  we assume $R_{1\bar 12\bar 2}=0 $ for a pair of unit $(1,0)$-vectors $e_1\perp e_2$. Then we need to show  $G_{1\bar 12\bar 2}\ge 0$.  \\

\noindent
{\bf Claim 2.3.} If $e_i \perp e_1$, then $R_{1\bar 1 2\bar i}=0$, similarly,  if $e_i \perp e_2$, then $R_{2\bar 2 1\bar i}=0.$ 

The first statement in Claim 2.3 follows from the simple fact that  if $e_i\perp e_1$, then $Rm(e_1,\overline{e_1},e_2 +s e_i,\overline{e_2 +s e_i})\ge 0$ for arbitray complex number $s$. 
The proof of second statement is similar. \\

\noindent
{\bf Claim 2.4} $R_{1\bar 2 1 \bar 1}=R_{1\bar 2 2\bar 2}.$ 

Note that $(e_1+s e_2 )\perp (e_2 -\bar s e_1)$ for any complex number $s$, hence $$Rm(e_1+s e_2,\overline{e_1+s e_2}, e_2 -\bar s e_1,\overline{e_2} -s \overline{e_1})\ge 0.$$
Again its first order derivative vanishes at point $s=0$, and Claim 2.4 follows. \\

\noindent
{\bf Claim 2.5.} $G_{1\bar 1 2\bar 2}=R_{1\bar 1 i\bar j}R_{j\bar i 2\bar 2}-|R_{1\bar i 2\bar j}|^2+|R_{1\bar 2\mu \bar \nu}|^2$, where $3\le i, j\le n$ and $1\le \mu,\nu \le n$.

In fact, from the definition of $G_{1\bar 1 2\bar 2}$, the assumption that $R_{1\bar 1 2\bar 2}=0$ and the above two claims, we have:
\begin{align*}
G_{1\bar 1 2\bar 2}=&R_{1 \bar 2 \mu \bar \nu } R_{\nu \bar \mu 2 \bar 1} + R_{1\bar 1 \mu \bar \nu } R_{\nu \bar \mu 2 \bar 2}-R_{1\bar \mu 2 \bar \nu } R_{\mu \bar 2 \nu \bar 1}\\
=&R_{1\bar 2 \mu \bar \nu } R_{\nu \bar \mu 2 \bar 1}\\
&+ R_{1\bar 1 i \bar j } R_{j \bar i 2 \bar 2}+ R_{1\bar 1 1 \bar 2 } R_{2 \bar 1 2 \bar 2}+ R_{1\bar 1 2 \bar 1 } R_{1\bar 2 2 \bar 2}\\
&-R_{1\bar i 2 \bar j } R_{i \bar 2 j \bar 1}-R_{1\bar 1 2 \bar 1 } R_{1 \bar 2 1 \bar 1}-R_{1\bar 2 2 \bar 2 } R_{2\bar 2 2 \bar 1}\\
=&R_{1\bar 1 i\bar j}R_{j\bar i 2\bar 2}-|R_{1\bar i 2\bar j}|^2+|R_{1\bar 2 \mu \bar \nu}|^2.
\end{align*}

Now for arbitrary $(1,0)$-vectors $X,Y\perp e_1,e_2$ and real number $s$, we have the following:
$$\left(e_\alpha+sX\right) \perp \left(e_\beta+sY-s^2e_\alpha<\bar X,Y>\right).$$
Thus using Claim 2.3, we have 
\begin{align*}
0\le & Rm(e_1+sX,\overline{e_1}+s \bar X,e_2+s Y- s^2 e_1 <\bar X,Y> , \overline{e_2}+s \bar Y-s^2 \overline{e_1}<X,\bar Y>)\\
=&s^2 \left( R_{2\bar 2 X\bar X}+R_{1\bar 1 Y\bar Y} +2 Re R_{X\bar 1 Y\bar 2}+2Re(R_{X\bar Y 2\bar 1}-R_{1\bar 12\bar 1}<X,\bar Y>)\right)+O(s^3)
\end{align*}
Hence, for all $s$, $X$ and $Y$, 
\begin{equation*}
\left( R_{2\bar 2 X\bar X}+R_{1\bar 1 Y\bar Y} +2 Re R_{X\bar 1 Y\bar 2}+2Re(R_{X\bar Y 2\bar 1}-R_{1\bar 12\bar 1}<X,\bar Y>)\right)\ge 0
\end{equation*}
By using Corollary 5.2.13 again, we obtain
\begin{equation*}
R_{1\bar 1 i\bar j} R_{j\bar i 2\bar 2}\ge |R_{i\bar 1 j\bar 2}|^2+|R_{i\bar j 2\bar 1}-R_{1\bar 12\bar 1}g_{i\bar j}|^2.
\end{equation*}
This together with Claim 2.5 implies that $G_{1\bar 1 2\bar 2}\ge 0$. The proof of Theorem 2.3 is completed.

\end{pf}

\begin{remark} Wilking \cite{Wi10} has provided a nice Lie algebra approach treating all known nonnegativity curvature conditions preserved
under the Ricci flow and KRF so far, including the nonnegative bisectional curvature and the nonnegative orthogonal bisectional curvature. 
\end{remark}

\section{The Li-Yau-Hamilton inequalities for KRF}

In \cite{LY}, Li-Yau developed a fundamental gradient estimate, now called Li-Yau estimate (aka differential Harnack inequality),
for positive solutions to the heat equation on a complete Riemannian manifold with nonnegative Ricci
curvature. They used it to derive the Harnack inequality for such
solutions by a path integration. Shortly after, based on a suggestion of
Yau, Hamilton \cite{Ha88} derived a similar estimate for the
scalar curvature of solutions to the Ricci flow on  Riemann
surfaces with positive curvature. Hamilton subsequently found a matrix
version of the Li-Yau estimate \cite{Ha93} for solutions to the
Ricci flow with positive curvature operator in all dimensions.  This matrix version of the Li-Yau estimate is now
called \textbf{Li-Yau-Hamilton estimate}, and it played a central role in the analysis of formation of singularities and the
application of the Ricci flow to three-manifold topology. Around the same time,
the author obtained the (matrix) Li-Yau-Hamilton estimate for the KRF with nonnegative bisectional curvature and 
the Harnack inequality for the evolving scalar curvature, as well as the determinant of the Ricci tensor,  by a similar path integration argument. 
We remark that our Li-Yau-Hamilton estimate for the KRF in the noncompact case played a crucial role in the works of Chen-Tang-Zhu \cite{CTZ04}, Ni \cite{Ni05}, 
Chau-Tam\cite{ChauTam06}, etc.  The presentation below essentially follows the original papers of Hamilton \cite{Ha88, Ha93, Ha93E} and the author \cite{Cao92, Cao97}.

We shall start by recalling the well-known Li-Yau inequality for positive solutions to the heat equation on
complete Riemannian manifolds with nonnegative Ricci curvature, and the important observation that Li-Yau
inequality becomes equality on the standard heat kernel on the Euclidean space. Then, following Hamilton, we show how one could derive the matrix Li-Yau-Hamilton quadratic for the KRF from the equation of expanding K\"ahler-Ricci solitons. Finally we state and sketch the matrix Li-Yau-Hamilton inequality for the KRF with nonnegative bisectional curvature. \\

\noindent{\bf 3.1 The Li-Yau estimate for the 2-dimensional Ricci flow}

\medskip

Let us begin by describing the Li-Yau estimate \cite{LY} for
positive solutions to the heat equation on a complete Riemannian
manifold with nonnegative Ricci curvature.

\begin{theorem}[{Li-Yau \cite{LY}}]
Let $(M, g_{ij})$ be an $n$-dimensional complete Riemannian
manifold with nonnegative Ricci curvature. Let $u(x,t)$ be any
positive solution to the heat equation
$$
\frac{\partial u}{\partial t}=\Delta u \qquad {\text{on}} \ \
M\times [0,\infty).
$$
Then, for all $t>0$, we have
$$\frac{\partial u}{\partial t}-\frac{|\nabla
u|^2}{u}+\frac{n}{2t}u\ge 0 \qquad \text{on } \ \ M\times
(0,\infty).\eqno(3.1)
$$
\end{theorem}

\smallskip
We remark that, as observed by Hamilton (cf. \cite{Ha93}), one can
in fact prove that for any vector field $V^i$ on $M$,
$$\frac{\partial u}{\partial t}+2\nabla u\cdot V+u|V|^2
+\frac{n}{2t}u\ge 0. \eqno(3.2)
$$
If we take the optimal vector field $V=-\nabla u/u$, then we
recover the inequality (3.1).

Now we consider the Ricci flow on a Riemann surface. Since in
(real) dimension two the Ricci curvature is given by
$$
R_{ij}=\frac{1}{2} Rg_{ij},
$$
the Ricci flow becomes
$$ \frac{\partial g_{ij}}{\partial t}=-R g_{ij}.\eqno(3.3)
$$

Now let $g_{ij} (t)$ be a complete solution of the Ricci flow
(3.3) on a Riemann surface $M$ and $0\le t<T$. Then the scalar
curvature $ R$ evolves by the semilinear equation
$$
\frac{\partial R}{\partial t}=\triangle R+R^{2}
$$
on $M\times[0,T)$. Suppose the scalar curvature of the initial
metric is bounded, nonnegative everywhere and positive somewhere.
Then it follows from the maximum principle that the scalar curvature
$R(x,t)$ of the evolving metric remains nonnegative. Moreover,
from the standard strong maximum principle (which works in each
local coordinate neighborhood), the scalar curvature is positive
everywhere for $t>0$. In \cite{Ha88}, Hamilton obtained the
following Li-Yau estimate for the scalar curvature $R(x,t)$.

\begin{theorem}[{Hamilton \cite{Ha88}}]
Let $g_{ij}(t)$ be a complete solution to the Ricci flow on a
surface $M$. Assume the scalar curvature of the initial metric is
bounded, nonnegative everywhere and positive somewhere. Then the
scalar curvature $R(x,t)$ satisfies the Li-Yau estimate
$$\frac{\partial R}{\partial t} - \frac{|\nabla R|^2}{R} +
\frac{R}{t} \geq 0.\eqno(3.4)
$$
\end{theorem}

\begin{pf}
By the above discussion, we know $R(x,t)>0$ for $t>0$. If we set
$$
L=\log R(x,t)\quad \text{for}\quad t>0,
$$
then
\begin{align*}
\frac{\partial}{\partial t}L&  = \frac{1}{R}(\triangle
R+R^{2})\\
&  = \triangle L+|\nabla L|^{2}+R
\end{align*}
and (3.4) is equivalent to
$$
\frac{\partial L}{\partial t}-|\nabla L|^{2}+\frac{1}{t}
=\triangle L +R +\frac{1}{t}\geq0.
$$

Following Li-Yau \cite{LY} in the linear heat equation case, we
consider the quantity
$$Q =\frac{\partial L}{\partial t}-|\nabla
L|^{2}=\triangle L+R.
$$
Then by a direct computation,
\begin{align*}
\frac{\partial Q}{\partial t}
&  = \frac{\partial}{\partial t}(\triangle L+R)\\
&  = \triangle\(\frac{\partial L}{\partial t}\)+R\triangle
L+\frac{\partial R}{\partial t}\\
&  = \triangle Q+2\nabla L\cdot\nabla Q+2|\nabla ^{2}L|^{2}
+2R(\triangle L)+R^{2}\\
&  \geq \triangle Q+2\nabla L\cdot\nabla Q+Q^{2}.
\end{align*}
So we get
$$
\frac{\partial}{\partial t}
\(Q+\frac{1}{t}\)\geq\triangle\(Q+\frac{1}{t}\)+2\nabla
L\cdot\nabla \(Q+\frac{1}{t}\)+\(Q-\frac{1}{t}\)\(Q+\frac{1}{t}\).
$$
Hence by the maximum principle argument, we obtain
$$
Q+\frac{1}{t}\ge 0.
$$
This proves the theorem.
\end{pf}

\noindent {\bf 3.2  Li-Yau estimate and expanding solitons}

\medskip
To prove inequality (3.4) for the scalar curvature
of solutions to the Ricci flow in higher dimensions is not so
simple. It turns out that one does not get inequality (3.4) directly,
but rather indirectly as the trace of certain matrix estimate when either curvature operator (in the Riemannian case) or bisectional curvature (in the K\"ahler case)
is nonnegative. The key ingredient in formulating this matrix version is an important observation by
Hamilton that the Li-Yau inequality, as well as its matrix version, becomes equality on the expanding solitons which he first discovered for
the case of the heat equation on ${\mathbb R}^n$. This led him and the author to formulate and prove
the matrix differential Harnack inequality,
now called Li-Yau-Hamilton estimates, for the Ricci flow in higher dimensions \cite{Ha93, Ha93E}
and the K\"ahler-Ricci flow \cite{Cao92, Cao97} respectively.

To illustrate, let us examine the heat equation case first. Consider the heat kernel
$$
u(x,t)=(4\pi t)^{-n/2} e^{-|x|^2/4t}, \eqno (3.5)
$$
which can be considered as an expanding soliton solution
for the standard heat equation on $\mathbb R^n$.

Differentiating the function $u$ once, we get
$$ \nabla_j u=-u
\frac{x_j}{2t} \; \text{ or }\;
\nabla_j u+uV_j=0, \eqno (3.6)
$$
where
$$
V_j=\frac{x_j}{2t}=-\frac{\nabla_j u}{u}.
$$

Differentiating (3.6) again, we have
$$\nabla_i\nabla_j
u+\nabla_iuV_j+\frac{u}{2t}\delta_{ij}=0.
\eqno(3.7)
$$ To make the expression in (3.7) symmetric in $i, j$, we
multiply $V_i$ to (3.6) and add to (3.7) and obtain
$$
\nabla_i\nabla_j u+\nabla_iuV_j+\nabla_juV_i+uV_iV_j
+\frac{u}{2t}\delta_{ij}=0. \eqno (3.8)
$$ Taking the trace in (3.8) and using the equation ${\partial
u}/{\partial t}=\Delta u$, we arrive at
$$
\frac{\partial u}{\partial t} +2\nabla u\cdot V+u|V|^2
+\frac{n}{2t} u=0,
$$
which shows that the Li-Yau inequality (3.1) becomes an equality
on our expanding soliton solution $u$! Moreover, we even have the
matrix identity (3.8).

Based on the above observation and in a similar process,
Hamilton \cite{Ha93} found a matrix quantity, which vanishes on expanding
gradient Ricci solitons and is nonnegative for any solution to the
Ricci flow with nonnegative curvature operator.  At the same time, the author \cite{Cao92} (see also \cite{Cao97}) proved the Li-Yau-Hamilton estimate
for the K\"ahler-Ricci flow with nonnegative bisectional curvature, see below.

To formulate the Li-Yau-Hamilton quadric, let us consider a homothetically expanding gradient K\"ahler-Ricci soliton $g$
satisfying
$$
R_{i\bar j}+\frac{1}{t}g_{i\bar j}=\nabla_i V_{\bar j}, \qquad \nabla_i V_{j} =0 \eqno(3.9)
$$
with $V_i=\nabla_if$ for some real-valued smooth function $f$ on $X$. Differentiating (3.9) and commuting
give the first order relations
$$\nabla_k R_{i\bar j} = \nabla_k\nabla_{\bar j} V_i-\nabla_{\bar j}\nabla_k V_i = -R_{k\bar j i\bar p} V_p,
$$  or
$$\nabla_k R_{i\bar j} + R_{i\bar j k\bar p} V_p=0,
\eqno(3.10)
$$
and
$$\nabla_k R_{i\bar j}  V_{\bar k}+ R_{i\bar j k\bar p} V_p V_{\bar k}=0. \eqno(3.11)
$$
Differentiating (3.10) again and using the first equation in (3.9), we get

$$\nabla_{\bar l}\nabla_k R_{i\bar j} + \nabla_{\bar p} R_{i\bar j k\bar l} V_p+R_{i\bar j k\bar p} R_{p\bar l}+\frac{1}{t} R_{i\bar j k\bar l}=0. \eqno(3.12) $$
Taking the trace in (3.12),  we get
$$\Delta R_{i\bar j} +\nabla_{\bar k} R_{i\bar j} V_{k} +R_{i\bar j k\bar l} R_{l\bar k}+\frac{1}{t} R_{i\bar j}=0. \eqno (3.13)$$
Symmetrizing by adding (3.11) to (3.13), we arrive at
$$\Delta R_{i\bar j} + \nabla_{k} R_{i\bar j} V_{\bar k}+ \nabla_{\bar k} R_{i\bar j} V_{k}+R_{i\bar j k\bar l} R_{l\bar k}+R_{i\bar j k\bar l} V_{l} V_{\bar k} + \frac{1}{t} R_{i\bar j}=0, $$
or,  by (2.9),  equivalently
$$\frac{\partial} {\partial t} R_{i\bar j} + \nabla_{k} R_{i\bar j} V_{\bar k}+ \nabla_{\bar k} R_{i\bar j} V_{k}+R_{i\bar k} R_{k\bar j}+R_{i\bar j k\bar l} V_{l} V_{\bar k} + \frac{1}{t} R_{i\bar j}=0. \eqno (3.14) $$

\medskip
\noindent {\bf 3.3 The Li-Yau-Hamilton estimates and Harnack's inequalities for KRF}

\medskip
We now state the Li-Yau-Hamilton estimates and the Harnack inequalities for KRF and NKRF
with nonnegative holomorphic bisectional curvature.

\begin{theorem}[{Cao \cite{Cao92, Cao97}}]
Let $g_{i \bar j}(t)$ be a complete solution to the
K\"ahler-Ricci flow on $X^n$ with bounded
curvature and nonnegtive bisectional curvature and $0\leq t < T$.
For any point $x\in X$ and any vector $V$ in the holomorphic
tangent space $T^{1,0}_x X$, let
$$
Z_{i \bar j}=\frac{\partial}{\partial t}R_{i\bar j}+ R_{i \bar k}R_{k\bar j}
+\nabla_{k}R_{i\bar j} V^{k}
+\nabla_{\bar k}R_{i\bar j}V^{\bar k}+R_{i\bar j k \bar \ell}V^{
k}V^{\bar\ell}+\frac {1}{t}R_{i\bar j} .
$$
Then we have
$$
Z_{i\bar j}W^{i}W^{\bar j}\geq 0
$$
for all $x\in X$, $V, W\in T^{1,0}_x X$, and $t>0$.
\end{theorem}

The proof of Theorem 3.3 is based on Hamilton's strong tensor maximum principle and involves a large amount of calculations. We refer the interested reader
to the original papers \cite{Cao92, Cao97} for details.

\begin{corollary}[{Cao \cite{Cao92, Cao97}}]
Under the assumptions of Theorem 3.3,  the scalar curvature $R$ satisfies the estimate
$$ \frac{\partial R}{\partial t} +\nabla_{i}R V^{i}+ \nabla_{\bar i}R V^{\bar i} +R_{i\bar j} V^{i} V^{\bar j} +\frac {R} {t}\ge 0 \eqno(3.15)
$$ for all $x\in X$ and $t>0$. In particular,
$$\frac{\partial R}{\partial t}-\frac{|\nabla R|^2}{R}
+\frac{R}{t}\geq 0. \eqno(3.16)$$
\end{corollary}

\begin{pf} The first inequality (3.15) follows by taking the trace of $Z_{i\bar j}$ in Theorem 3.3. By taking $V=-\nabla \log R$ in (3.15) and
observing $R_{i\bar j}\le R g_{i\bar j}$ (because $R_{i\bar j}\ge 0$),
we obtain the second inequality  (3.16).
\end{pf}

\medskip

As a consequence of Corollary 3.1, we obtain the following Harnack
inequality for the scalar curvature $R$ by taking the Li-Yau type
path integral as in \cite{LY}.

\begin{corollary}[{Cao \cite{Cao92, Cao97}}]
Let $g_{i\bar j}(t)$ be a complete solution to the K\"ahler-Ricci flow on $X^n$ with bounded and nonnegative bisectional curvature.
Then for any points $x_{1},x_{2}\in X$, and $0<t_{1}<t_{2}$, we have
$$
R(x_1,t_1)
\leq \frac{t_2}{t_1} e^{d_{t_1} {(x_1,x_2)}^2/{4(t_2-t_1)}}R(x_2,t_2).
 $$
Here $d_{t_1} (x_1,x_2)$ denotes the distance between $x_1$ and $x_2$ with respect to $g_{i\bar j} (t_1)$.
\end{corollary}

\begin{pf}
Take the geodesic path $\gamma(\tau)$, $\tau\in[t_1,t_2],$ from
$x_1$ to $ x_2$ at time $t_1$ with constant velocity
$d_{t_1}(x_1,x_2)/(t_2-t_1).$ Consider the space-time path
$\eta(\tau)=(\gamma(\tau),\tau)$, $\tau\in[t_1,t_2]$. We compute
\begin{align*}
\log\frac{R(x_2,t_2)}{R(x_1,t_1)}
&  = \int^{t_2}_{t_1}\frac{d}{d\tau}\log R  (\gamma(\tau),\tau)d\tau\\
&  = \int^{t_2}_{t_1}\frac{1}{R}\(\frac{\partial R}{\partial \tau}
+\nabla R \cdot \frac{d\gamma}{d\tau}\)d\tau\\
&  \geq \int^{t_2}_{t_1}\(\frac{\partial \log R}{\partial\tau} -|\nabla
\log R|^2_{g (\tau)}
-\frac{1}{4}\left|\frac{d\gamma}{d\tau}\right|^2_{g (\tau)}\)d\tau.
\end{align*}
Then, by the Li-Yau estimate (3.16) for $R$ in Corollary 3.1 and the fact that the metric is shrinking
(since the Ricci curvature is nonnegative), we have
\begin{align*}
\log\frac{R(x_2,t_2)}{R(x_1,t_1)} &
\geq\int^{t_2}_{t_1}\(-\frac{1}{\tau}-
\frac{1}{4}\left|\frac{d\gamma}{d\tau}\right|^2_{g (t_1)}\)d\tau \\
&  = \log\frac{t_1}{t_2}-\frac{d_{t_1}(x_1,x_2)^2}{4(t_2-t_1)}.
\end{align*}
Now the desired Harnack inequality follows by exponentiating.
\end{pf}

Finally,  we can convert Corollary 3.1 and Corollary 3.2 to the NKRF case and yield the following Li-Yau type estimate and Harnack's inequality.

\begin{theorem} [{Cao \cite{Cao92}}]Let $g_{i\bar j} (t)$ be a solution to NKRF on $X^n \times [0, \infty)$ with nonnegative bisectional curvature. Then,
the scalar curvature $R$ satisfies

\begin{itemize}
\item[(a)]  the Li-Yau type estimate: for any $t>0$ and $x\in X$,
$$
\frac{\partial R}{\partial t}-\frac{|\nabla R|^2}{R}
+\frac{R}{1-e^{-t}}\geq 0; \eqno (3.17)$$

\item[(b)]  the Harnack inequality: for any $0<t_1<t_2$ and any $x, y\in X$,
$$ R(x,t_1)
\leq \frac{e^{t_2}-1}{e^{t_1}-1}  \exp \{ e^{t_2-t_1} \frac {d^2_{t_1} {(x, y)}} {4(t_2-t_1)} \}
R(y,t_2),
\eqno (3.18)
$$
\end{itemize}
\end{theorem}

\begin{pf} Part (a):
Let $\hat g_{i\bar j} (s)$ be the associated solution to KRF on $X\times [0,1)$. By Lemma 2.3, Corollary 2.1 and Corollary 3.1, we have
$$ R = (1-s) \hat R,  \qquad1-e^{-t}=s, $$ and
$$\frac{\partial \hat R}{\partial s}-\frac{|\nabla \hat R|_{\hat g}^2}{\hat R}
+\frac{\hat R}{s}\geq 0.$$ It is then easy to check that they are translated into (3.17).

Part (b): By the Li-Yau path integration
argument as in the proof of Corollary 3.2 but use (3.17) instead,  we get
\begin{align*}
\log\frac{R(y, t_2)}{R(x, t_1)} &
\geq\int^{t_2}_{t_1}\(-\frac{1}{1-e^{-\tau}}-
\frac{1}{4}\left|\frac{d\gamma}{d\tau}\right|^2_{g (\tau)}\)d\tau \\
&  =  \log\frac{e^{t_1}-1}{e^{t_2}-1}-\frac {1} {4} \Delta (x, t_1; y, t_2).
\end{align*}
where $$\Delta (x, t_1; y, t_2)=\inf _{\gamma} \int ^{t_2}_{t_1} |\gamma'(\tau)|^2_{g(\tau)} d\tau.  \eqno (3.19)$$
But, the NKRF equation and the assumption of $Rc_{g}\ge 0$ imply that, for $t_1<t_2$, $$ g(t_2) \le e^{t_2-t_1} g(t_1).$$ Hence,
$$ \Delta (x, t_1; y, t_2) \le e^{t_2-t_1} \frac {d^2_{t_1} {(x, y)}} {(t_2-t_1)}.$$
Therefore,
$$ \log\frac{R(y, t_2)}{R(x, t_1)} \ge \log\frac{e^{t_1}-1}{e^{t_2}-1}  -  e^{t_2-t_1} \frac {d^2_{t_1} {(x, y)}} {4(t_2-t_1)}. $$
\end{pf}

\section{Perelman's $\mu$-entropy and $\kappa$-noncollapsing theorems}

In this section, we review Perelman's $\mathcal{W}$-functional and the associated $\mu$-entropy. We show that  the $\mu$-entropy is monotone under the Ricci flow and use this
important fact to prove a strong $\kappa$-noncollapsing theorem for the Ricci flow on compact Riemannian manifolds.  These results and the ideas in the proof play a crucial
role in the next two sections when we discuss the uniform estimates on the diameter and the scalar curvature of the NKFR.

\medskip
\noindent {\bf 4.1 Perelman's $\mathcal{W}$-functional and $\mu$-entropy for the Ricci flow}
\medskip

Let $M$ be a compact $n$-dimensional manifold. Consider the following functional, due to Perelman \cite{P1}, 
$$\mathcal{W}(g_{ij},f,\tau) =\int_M [\tau (R+|\nabla
f|^2)+f-n](4\pi
\tau)^{-\frac{n}{2}}e^{-f}dV \eqno(4.1)$$
under the constraint $$(4\pi\tau)^{-\frac{n}{2}}\int_M e^{-f}dV=1. \eqno(4.2)$$
Here $g_{ij}$ is any given Riemannian metric, $f$ is any smooth
function on $M$, and $\tau$ is a positive scale parameter. Clearly
the functional $\mathcal{W}$ is invariant under simultaneous
scaling of $\tau$ and $g_{ij}$ (or equivalently the parabolic
scaling), and invariant under diffeomorphism. Namely, for any
positive number $a>0$ and any diffeomorphism $\varphi\in {\rm Diff }(M^n)$, 
$$\mathcal{W}(\varphi^*g_{ij}, \varphi^*f,
\tau)=\mathcal{W}(g_{ij},f,\tau) \quad \mbox{and} \quad \mathcal{W}(a g_{ij}, f, a
\tau)=\mathcal{W}(g_{ij},f,\tau).\eqno (4.3)$$

In \cite{P1} Perelman derived the following first variation formula (see also \cite{CaoZhu06})

\begin{lemma}[{Perelman \cite{P1}}]
If $v_{ij}=\delta g_{ij},\; h=\delta f,\;\mbox{ and }\;
\eta=\delta\tau$, then
\begin{align*}
& \delta \mathcal{W}(v_{ij},h,\eta)\\
&  = \int_M-\tau v_{ij}\(R_{ij}+\nabla_i\nabla_jf
-\frac{1}{2\tau}g_{ij}\)(4\pi\tau)^{-\frac{n}{2}}e^{-f}dV\\
&\quad +\int_M\(\frac{v}{2}-h-\frac{n}{2\tau}\eta\)[\tau(R+2\Delta f
-|\nabla f|^2)+f-n-1](4\pi\tau)^{-\frac{n}{2}}e^{-f}dV\\
&\quad +\int_M \eta\(R+|\nabla f|^2
-\frac{n}{2\tau}\)(4\pi\tau)^{-\frac{n}{2}}e^{-f}dV.
\end{align*}
Here $v=g^{ij}v_{ij}$.
\end{lemma}

By Lemma 4.1 and direct computations (cf. \cite{P1, CaoZhu06}), one obtains

\begin{lemma}[{Perelman \cite{P1}}]
If $g_{ij}(t), f(t)$ and $\tau(t)$ evolve according to the system
$$
      \left\{
       \begin{array}{lll}
       \displaystyle
\frac{\partial g_{ij}}{\partial t}=-2R_{ij},
          \\[4mm]
      \displaystyle
\frac{\partial f}{\partial t}=-\Delta f+|\nabla
f|^2-R+\frac{n}{2\tau},
          \\[4mm]
      \displaystyle
\frac{\partial \tau}{\partial t}=-1,
       \end{array}
      \right.
$$
then
$$
\frac{d} {d t} \mathcal{W}(g_{ij}(t),f(t),\tau(t))=\int_M
2\tau\left|R_{ij}+\nabla_i\nabla_jf-\frac
{1}{2\tau}g_{ij}\right|^2 (4\pi\tau)^{-\frac{n}{2}} e^{-f}dV,
$$
and $\int_M(4\pi\tau)^{-\frac{n}{2}}e^{-f}dV$ is constant. In
particular $\mathcal{W}(g_{ij}(t),f(t),\tau(t))$ is nondecreasing
in time and the monotonicity is strict unless we are on a
shrinking gradient soliton.
\end{lemma}

Now we define
$$\mu(g_{ij},\tau)=\inf\left\{\mathcal{W}(g_{ij},f,\tau)\ |\  f\in
C^\infty(M),
\frac{1}{(4\pi\tau)^{n/2}}\int_M e^{-f}dV=1\right\}.  \eqno(4.4) $$

Note that if we set $u=e^{-f/2}$, then the functional
$\mathcal{W}$ can be expressed as
$$
\mathcal{W}=\mathcal{W}(g_{ij},u,\tau) = (4\pi\tau)^{-\frac{n}{2}}\int_M [\tau(Ru^2+4|\nabla
u|^2)-u^2\log u^2-nu^2] dV \eqno(4.5)
$$
and the constraint (4.2)
becomes $$(4\pi\tau)^{-\frac{n}{2}}\int_M u^2dV=1. \eqno(4.6)$$
Thus $\mu(g_{ij},\tau)$ corresponds to the best constant of a
logarithmic Sobolev inequality. Since the non-quadratic term is
subcritical (in view of Sobolev exponent), it is rather
straightforward to show that
$$
\inf\!\bigg\{\! (4\pi\tau)^{-\frac{n}{2}}\!\int_M\! [\tau(4|\nabla u|^2\!+Ru^2) -u^2\log
u^2\!-nu^2] dV:  (4\pi\tau)^{-\frac{n}{2}}
\int_Mu^2dV\!=\!1\!\bigg\}
$$
is achieved by some nonnegative function $u\in H^1(M)$ which
satisfies the Euler-Lagrange equation
$$\tau (-4\Delta u+Ru)-2u\log u-nu=\mu(g_{ij},\tau)u
.$$ One can further show that $u$ is positive (see \cite{Ro}).
Then the standard regularity theory of elliptic PDEs shows that
$u$ is smooth. We refer the reader to Rothaus \cite{Ro} for more
details. It follows that $\mu(g_{ij},\tau)$ is achieved by a
minimizer $f$ satisfying the nonlinear equation 
$$\tau (2\Delta f-|\nabla f|^2+R)+f-n=\mu(g_{ij},\tau). \eqno(4.7) $$

It turns out that the $\mu$-entropy has the following important monotonicity property under the Ricci flow:
\begin{proposition} [{Perelman \cite{P1}}]\  Let $g_{ij}(t)$ be a solution to the Ricci flow
$$\frac{\partial g_{ij}}{\partial t}=-2R_{ij}$$
on $M^n\times [0,T)$ with $0<T<\infty$, then $\mu(g_{ij}(t),T_{0}-t)$
is nondecreasing along the Ricci flow for any $T_{0}\ge T$; moveover, the monotonicity is strict unless we are on
a shrinking gradient soliton.
\end{proposition}

\begin{pf}
Fix any time $t_0$, let $f_0$ be a minimizer of
$\mu(g_{ij}(t_0),T_{0}-t_0).$ Note that the backward heat equation
$$
\frac{\partial f}{\partial t}=-\Delta f+|\nabla
f|^2-R+\frac{n}{2\tau}
$$
is equivalent to the linear equation
$$
\frac{\partial }{\partial
t}((4\pi\tau)^{-\frac{n}{2}}e^{-f})=-\Delta
((4\pi\tau)^{-\frac{n}{2}}e^{-f})+R((4\pi\tau)^{-\frac{n}{2}}e^{-f}).
$$
Thus we can solve the backward heat equation of $f$ with
$f|_{t=t_0}=f_0$ to obtain $f(t)$ for $t\in [0, t_0]$, satisfying constraint (4.2). Then, for $t\leq t_0$, it follows
from Lemma 4.2  that
\begin{align*}
\mu(g_{ij}(t),T_{0}-t)
& \leq \mathcal{W}(g_{ij}(t),f(t),T_{0}-t)\\
& \leq \mathcal{W}(g_{ij}(t_0),f(t_0),T_{0}-t_0)\\
& = \mu(g_{ij}(t_0),T_{0}-t_0),
\end{align*}
and the second inequality is strict unless we are
on a shrinking gradient soliton.
\end{pf}

\medskip

\noindent {\bf 4.2 Strong $\kappa$-noncollapsing of the Ricci flow}

\medskip
We now apply the monotonicity of the $\mu$-entropy in Proposition 4.1 to prove  a strong
version of Perelman's  {\bf no local collapsing theorem}, which is extremely important because it gives a local injectivity
radius estimate in terms of the local curvature bound.

\begin{definition}
Let $g_{ij}(t), 0 \leq t <T,$ be a solution to the Ricci flow on an $n$-dimensional
manifold $M$, and let $\kappa$, $r$ be two positive constants. We say that the solution $g_{ij}(t)$
is \textbf{$\kappa$-noncollapsed} at $(x_0,t_0)\in M \times [0,T)$ on
the scale $r$ \index{$\kappa$-noncollapsed} if we have
$${\rm V}_{t_0}(x_0,r)) \geq \kappa r^n,$$
whenever
$$
|Rm|(x, t_0) \leq r^{-2}
$$
for all $x \in B_{t_0}(x_0,r)$. Here $B_{t_0}(x_0,r)$ is
the geodesic ball centered at $x_0\in M $ and of radius $r$ with
respect to the metric $g_{ij}(t_0)$.
\end{definition}

\begin{remark} In \cite{P1}, Perelman also defined $\kappa$-noncollapsing by requiring the curvature bound
$|Rm|(x, t) \leq r^{-2}$
on the (backward) parabolic cylinder $B_{t_0}(x_0,r)\times [t_0-r^2, t_0]$.

\end{remark}

The following result was proved in \cite{CaoZhu06} (cf. Theorem 3.3.3 in \cite{CaoZhu06})).

\begin{theorem}[Strong no local collapsing theorem]
Let $M$ be a compact Riemannian manifold, and let $g_{ij}(t)$
be a solution to the Ricci flow on $M^n\times [0,T)$ with $0<T<+\infty$. Then there
exists a positive constant $\kappa$, depending only the initial
metric $g_0$ and  $T$, such that $g_{ij}(t)$ is $\kappa$-noncollapsed at very point $(x_0,t_0)\in M \times [0,T)$
on all scales less than $\sqrt{T}$. In fact, for any $(x_0,t_0)\in M \times [0,T)$ and $0<r\le \sqrt{T}$ we have
$${\rm V}_{t_0}(x_0,r) \geq \kappa r^n,$$ whenever
$$
R(\cdot, t_0) \leq r^{-2}  \quad \mbox{on} \ B_{t_0}(x_0, r).
$$
\end{theorem}

\begin{pf}
Recall that
$$
\mu(g_{ij},\tau)=\inf\left\{\mathcal{W}(g_{ij},u,\tau)\ \Big|\
\int_M(4\pi\tau)^{-\frac{n}{2}}u^2dV=1\right\}.
$$
where,
$$
\mathcal{W}(g_{ij},u,\tau) = (4\pi\tau)^{-\frac{n}{2}} \int_M[\tau(Ru^2+4|\nabla
u|^2)-u^2\log u^2-nu^2] dV.
$$

Set
$$
\mu_0=\inf\limits_{0\leq\tau\leq2T}\mu(g_{ij}(0),\tau)>-\infty. \eqno(4.8)
$$
By the monotonicity of $\mu(g_{ij}(t),\tau-t)$ in Proposition 4.1, we have
$$\mu_0 \le \mu(g_{ij}(0),t_0+r^2)\le \mu(g_{ij}(t_0),r^2) \eqno(4.9)$$
for $t_0<T$ and $r^2\leq T$.

Take a smooth cut-off function $\zeta (s)$, $0\le \zeta\leq1$, such that
$$ {\zeta (s) =\left\{
       \begin{array}{ll}
  1, \ \ \qquad  |s|\le 1/2,\\[4mm]
  0, \ \ \qquad  |s|\ge 1
       \end{array}
    \right.}$$
and $|\zeta'|\leq 2$ everywhere. Define a test function $u(x)$ on $M$ by
$$
u(x)=e^{L/2}\zeta\(\frac{d_{t_0}(x_0, x)}{r}\),
$$
where the constant $L$ is chosen so that
$$(4\pi r^2)^{-\frac{n}{2}}\int_M u^2 dV_{t_0}=1$$
Note that
$$|\nabla u|^2= e^{L}r^{-2}|\zeta'(\frac{d_{t_0}(x_0, x)}{r})|^2 \quad \mbox{and} \quad u^2\log u^2=Lu^2 +e^{L}\zeta^2\log \zeta^2.$$
Also, by the definition of $u(x)$, we have
$$ (4\pi r^2)^{-\frac{n}{2}} e^{L} {\rm V}_{t_0}(x_0, r/2) \le 1, \eqno(4.10)$$ and
$$ (4\pi )^{-\frac{n}{2}} r^{-n}e^{L} {\rm V}_{t_0}(x_0, r) \ge 1. \eqno(4.11)$$

Now it follows from (4.9) and the upper bound assumption on $R$ that

\begin{align*} \mu_0 &  \le \cal W (g_{ij}(t_0), u, r^2) \\
& =  (4\pi r^2)^{-\frac{n}{2}}\int_M [r^2 (Ru^2+4|\nabla u|^2)-u^2\log
u^2-nu^2]\\
& \le 1-L-n + (4\pi r^2)^{-\frac{n}{2}}e^{L}\int_M
(4|\zeta'|^2-\zeta^2\log\zeta^2)\\
& \le 1-L-n + (4\pi r^2)^{-\frac{n}{2}}e^{L} (16+e^{-1}){\rm V}_{t_0}(x_0, r).\\
\end{align*}
Here, in the last inequality, we have used the elementary fact that $-s\log s\le e^{-1}$
for $0\le s\le 1$.
Combining the above with (4.10), we arrive at
$$\mu_0 \le 1-L-n + (16+e^{-1})\frac{{\rm V}_{t_0}(x_0, r)} {{\rm V}_{t_0}(x_0, r/2)}. \eqno(4.12)$$

Notice that if we have the volume doubling property
$$ {\rm V}_{t_0}(x_0, r)\le C {\rm V}_{t_0}(x_0, r/2)$$ for some universal constant $C>0$, then (4.11) and (4.12) together would imply
$${\rm V}_{t_0}(x_0, r) \ge \exp \{\mu_0+n-1-(16+e^{-1})C\}  r^n, \eqno(4.13)$$ thus proving the theorem.  We now describe how to bypass such a volume doubling property by
a clever argument\footnote{Perelman also used a somewhat similar argument in proving his uniform diameter estimate for the NKRF, see the proof of Claim 1 in Section 6.}
pointed out by B.-L. Chen back in 2003.

Notice that the above argument is also valid if we replace $r$ by any positive number $0<a\le r$. Thus, at least we
have shown the following

\smallskip
\noindent {\bf Assertion}: Set
$$\kappa=\min\left\{ \exp [\mu_0+n-1-(16+e^{-1})3^n], \frac{1}{2}\alpha_n\right\},
$$
where $\alpha_n$ is the volume of the unit ball in $\mathbb{R}^n$.
Then, for any $0<a\le r$, we have
$$ {\rm V}_{t_0}(x_0,a)\geq\kappa a^n,  \eqno{(\ast)_a}
$$ whenever the volume doubling property, $${\rm
V}_{t_0}(x_0,a) \leq 3^{n}{\rm
V}_{t_0} (x_0,a/2), $$ holds.

\smallskip
Now we finish the proof by contradiction. Suppose $(\ast)_a$ fails for $a=r$. Then we must have
\begin{displaymath}
\begin{split}
  {\rm V}_{t_0}(x_0,\frac{r}{2})
&  < 3^{-n}V_{t_0}(x_0,r)\\
   &  < 3^{-n}\kappa r^n\\
   &  <\kappa\(\frac{r}{2}\)^n.
\end{split}
\end{displaymath}
This says that $(\ast)_{r/2}$ would also fail. By
induction, we deduce that
$$
{\rm V}_{t_0}(x_0,\frac{r}{2^k})<\kappa\(\frac{r}{2^k}\)^n\
\ \ \text{for all } k\geq1.
$$
But this contradicts the fact that
$$\lim\limits_{k\rightarrow\infty} \frac {
{\rm V}_{t_0}(x_0,\frac{r}{2^k})} {\(\frac{r}{2^k}\)^n}=\alpha_n.$$
\end{pf}

\medskip

{\bf 4.3 The $\mu$-entropy and the strong noncollapsing estimate for the NKRF}

\smallskip
\noindent To convert the $\kappa$-noncollapsing theorem for the Ricci flow to the KRF and NKRF,
first note that for any local holomorphic coordinates $(z^1, \cdots, z^n)$ with $z^i=x^i +\sqrt{-1}y^{i}$,  $(x^1, \cdots, x^n, y^1, \cdots, y^n)$ form a preferred
smooth local coordinates with
$$\frac {\partial} {\partial z^i}=\frac {1}{2} (\frac {\partial} {\partial x^i} -\sqrt {-1} \frac {\partial} {\partial y^i}) \qquad \mbox{and} \qquad
\frac {\partial} {\partial \bar z^i} =\frac 1 2 (\frac {\partial} {\partial x^i} +\sqrt {-1} \frac {\partial} {\partial y^i}).$$ Thus, in terms of the corresponding Riemannian metric
$ds^2$, we have

$$ds^2 (\frac {\partial}{\partial x^i}, \frac {\partial}{\partial x^j})=ds^2 (\frac {\partial}{\partial y^i}, \frac {\partial}{\partial y^j})=2 \Re (g_{i\bar j})$$ while
$$ ds^2 (\frac {\partial}{\partial x^i}, \frac {\partial}{\partial y^j})=2\Im  (g_{i\bar j}). $$ In particular,
for any $(z^1, \cdots, z^n)$ with $g_{i\bar j}=\delta_{i\bar j}$ (e.g., under normal  coordinates), then
$$ds^2 (\frac {\partial}{\partial x^i}, \frac {\partial}{\partial x^j})=ds^2 (\frac {\partial}{\partial y^i}, \frac {\partial}{\partial y^j})=2\delta_{ij} \quad \mbox{and} \quad ds^2 (\frac {\partial}{\partial x^i}, \frac {\partial}{\partial y^j})=0.$$
(Thus, we can symbolically express the Riemannian metric $g_{\mathbb R}=ds^2=2 g_{i\bar j}.$)

On the other hand, if $R_{i\bar j}=\lambda \delta_{i\bar j}$ under the normal holomorphic coordinates $(z^1, \cdots, z^n)$ then, for the Riemannian Ricci tensor $Rc_{ds^2}$, we have
$$Rc_{ds^2} (\frac {\partial}{\partial x^i}, \frac {\partial}{\partial x^j})=Rc_{ds^2}(\frac {\partial}{\partial y^i}, \frac {\partial}{\partial y^j})=2\lambda \delta_{ij} \quad \mbox{and} \quad
Rc_{ds^2} (\frac {\partial}{\partial x^i}, \frac {\partial}{\partial y^j})=0.$$  That is, $$Rc_{ds^2}=\lambda ds^2,$$ so we have the same Einstein constant $\lambda$.

Note that we also have the following relations:

\smallskip
$\bullet$ The scalar curvature: $R_{ds^2}=2R $

$\bullet$ The Laplace operator: $ \Delta_{ds^2}=2 \Delta$

$\bullet$ The norm square of the gradient of a function: $|\nabla f|^2_{ds^2}= 2 |\nabla f|^2$,   etc.

\smallskip
\noindent In particular, we have
$$R_{ds^2}+|\nabla f|^2_{ds^2}=2(R+|\nabla f|^2).$$

Therefore, with $\sigma=2\tau$, the Riemannian $\cal W$-functional on $(X^n, g_{i \bar j})$ is given by
$$\mathcal{W}= \frac{1} {(2\pi\sigma)^{n}}\!\int_X [\sigma (R+|\nabla f|^2) +f -2n] e^{-f} dV, \eqno (4.14)$$
or, with $u=e^{-f/2}$, by
$$\mathcal{W}(g_{i\bar j},u,\sigma)= \frac{1} {(2\pi\sigma)^{n}}\!\int_X [\sigma(Ru^2+4|\nabla
u|^2)-u^2\log u^2-2nu^2] dV \eqno (4.15)$$
with respect to the K\"ahler metric $g_{i\bar j}$.

The $\mu$-entropy is then given by
$$\mu= \mu(g_{i\bar j},\sigma)=\inf\left\{\mathcal{W}(g_{i\bar j},u,\sigma)\!:
(2\pi\sigma)^{-n}\int_X  u^2 dV=1\right\}.
$$

For any solution $\hat g_{i\bar j} (s)$ to the KRF on the maximal time interval $[0, 1)$, by taking $\sigma=1-s$, it follows that
$ \mu(\hat g_{i\bar j}(s),1-s)$ is monotone increasing in $s$. By the scaling invariance property of $\mu$ in (4.3) and the relation between
KRF and NKRF as described in Lemma 2.3, we get
$$\mu(\hat g_{i\bar j}(s),1-s)=\mu(g_{i\bar j} (t), 1). \eqno(4.16)$$

Thus,  by the monotonicity of $\mu(\hat g_{i\bar j}(s),1-s)$ and $ds/dt=e^{-t}>0$, we have

\begin{lemma} Let $g_{i\bar j}(t)$ be a solution to the NKRF on $X^n\times [0, \infty)$. Then,
\begin{align*}
\mu(g_{i\bar j} (t), 1) & =\inf\left\{ \frac {1} {(2\pi)^{n}}\! \int_X \(R+|\nabla f|^2 +f -2n\) e^{-f} dV\!:
\frac {1} {(2\pi)^{n}} \int_X e^{-f} dV=1\right\} \\
& =\inf\left\{ \frac {1} {(2\pi)^{n}}\! \int_X (Ru^2+4|\nabla
u|^2-u^2\log u^2-2nu^2)\!:
\frac {1} {(2\pi)^{n}} \int_X u^2=1\right\}
\end{align*}
 is monotone increasing in $t$.
\end{lemma}

Finally, we have the corresponding strong no local collapsing theorem for the NKRF:

\begin{theorem} [Strong no local collapsing theorem for NKRF] Let $X^n$ be a Fano manifold, and let $g_{i\bar j}(t)$
be a solution to the NKRF (2.5) on $X^n\times [0,\infty)$. Then there
exists a positive constant $\kappa>0$, depending only the initial
metric $g_0$, such that $g_{i\bar j}(t)$ is strongly $\kappa$-noncollapsed at very point $(x_0,t_0)\in M \times [0,\infty)$
on all scales less than $e^{t_0/2}$ in the following sense: for any $(x_0,t_0)\in X \times [0,\infty)$ and $0<r\le e^{t_0/2}$ we have
$$V_{t_0}(x_0,r) \geq \kappa r^{2n}, \eqno(4.17)$$ whenever
$$
R(\cdot, t_0) \leq r^{-2}  \quad \mbox{on} \ B_{t_0}(x_0, r). \eqno(4.18)
$$
\end{theorem}

\begin{pf} This is an immediate consequence of Theorem 4.1 applied to the KRF on $X^n\times [0,1)$,  and the relation between the KRF and the NKRF as described by Lemma 2.3.
\end{pf}

\bigskip
\section {Uniform curvature and diameter estimates for NKRF with nonnegative bisectional curvature}

Our goal in this section is to prove the uniform diameter and (scalar) curvature estimates by B.L Chen, X.-P. Zhu and the author \cite{CCZ03} for the NKRF
with nonnegative holomorphic bisectional curvature. The main ingredients of the proof are the Harnack estimate in Theorem 3.4 and the  strong
non-collapsing estimate in Theorem 4.2 for the NKRF.

\medskip

\begin{theorem} Let $(X^n, \tilde g_{i\bar j})$ be a compact K\"ahler manifold with nonnegative bisectional curvature and let $g_{i\bar j}(t)$ be the solution to
the NKRF with $g_{i\bar j}(0)=\tilde g_{i\bar j}$. Then, there exist positive constants $C_1>0$ and $C_2>0$ such that

\medskip

(i)  $|Rm|(x, t)\le C_1$ for all $(x, t)\in X\times [0,\infty)$;

\medskip

(ii) {\rm diam} \!$(X^n, g_{i\bar j}(t))\le C_2$ for all $t\ge 0$.
\end{theorem}

\begin{pf} By Theorem 2.2, we know that $g_{i\bar j}(t)$ has nonnegative bisectional curvature for all $t\ge 0$. Thus, it suffices to show the uniform upper bound for the scalar curvature
$$R(x, t)\le C_1 $$
on $  X\times [0,\infty).$ We divide the proof into several steps:
\medskip

{\bf Step 1: A local uniform bound on $R$}
\smallskip

First of all, we know that the volume $V_t(X^n)=\Vol (X, g_{i\bar j}(t))$ and the total scalar curvature $\int _{X^n} R(x, t) dV_t$ are constant along the NKRF. Hence the average scalar curvature
is also constant. In fact,
$$ \frac {1} {V_t(X^n)} \int _{X^n} R(x, t) dV_t =n, \quad \mbox{for all} \ t\ge 0.$$

Now, $\forall \ t>1$, set $t_1=t$, $t_2=t+1$ and pick a point $y_t\in X$ such that $$R(y_t, t+1)=n.$$ Then, $\forall \ x\in X$, by the Harnack inequality in Theorem 3.4,
and noting that $\forall t\ge 1$,
$$\frac{e^{t+1}-1}{ e^{t}-1} \le e+1,$$  we have
$$R(x, t) \le n (e+1)  \exp \(\frac {e} {4}  d^2_{t} {(x, y_t)}\).   \eqno (5.1)$$
In particular,
when $d_{t} (y_{t}, x)<1$, we obtain
a uniform upper bound
$$R(\cdot, t)\le n(e+1)\exp (e^2/4) \eqno(5.2)$$ on the unit geodesic ball $B_t(y_t, 1)$ at time $t$, for all $ t\ge 1$.

\medskip
{\bf Step 2: The uniform diameter bound}
\smallskip

Now we have the uniform upper bound (5.2) for the scalar curvature on $B_t (y_t, 1)$. By applying the strong no
local collapsing Theorem 4.2, there exists  a
positive constant $\kappa>0$, depending only on the initial metric $g_{0}$,  such that we have the following uniform lower bound
$${\rm V}_{t} (y_t, 1)\ge \kappa>0$$ for the volume of the unit geodesic ball $B_t (y_t, 1)$ for all $t\ge1$.

Suppose ${\rm diam} (X, g_{i\bar j}(t))$ is not uniformly bounded
from above in $t$. Then, there exist  a sequence of positive
numbers $\{D_k\}\to \infty$ and a time sequence $\{t_k\}\to
\infty$  such that
$${\rm diam} \ (X, g_{i\bar j}(t_k))>D_k. $$
However, since $g_{i\bar j}(t_k)$ has nonnegative Ricci curvature, it follows from an argument of Yau (cf. p.24 in \cite{ScY}) that there exists a
universal constant $C=C(n)>0$ such that
$${\rm V}_{t_k} (y_{t_k}, D_{k})\ge C {\rm V}_{t_k} (y_{t_k}, 1) D_{k} \ge \kappa C D_{k} \to \infty.$$
But this contradicts the fact that  $${\rm V}_{t_k} (y_{t_k}, D_{k})\le V_{t_k} (X^n)=V_0, \qquad k=1, 2, \cdots. $$ Thus, we have proved the uniform diameter bound:
there exists a positive constant $D>0$ such that for all $t>0$,
 $${\rm diam} \ (X, g_{i\bar j}(t))\le D. \eqno(5.3)$$

\medskip
{\bf Step 3: The global uniform bound on $R$}
\smallskip

Once we have the uniform diameter upper bound (5.3),  the Harnack inequality (5.1) immediately implies  the uniform scalar curvature upper bound,
$$R(x, t)\le n(e+1) e^{e D^2/4},$$
on $X^n\times [0, \infty)$.
\end{pf}

\medskip
\begin{remark}
As mentioned in the introduction, assuming in addition the existence of K-E metrics, Chen and Tian studied the NKRF with nonnegative bisectional curvature on Del Pezzo surfaces \cite{CTi02} and Fano 
manifolds in higher dimensions \cite{CTi06}.
\end{remark}

\section {Perelman's uniform scalar curvature and diameter estimates for NKRF}

In the previous section, we saw that when a solution $g_{i\bar j}(t)$ to the NKRF has nonnegative bisectional curvature, then the uniform diameter and curvature bounds follow from
a nice interplay between the Harnack inequality for the scalar curvature $R$ and the strong no local collapsing theorem.  In this section, we shall see Perelman's
amazing uniform estimates on the diameter and the scalar curvature for the NKRF on general Fano manifolds (Theorem 6.1). In absence of the Harnack inequality, Perelman's
proof is much more subtle, yet the monotonicity of the $\mu$-entropy and the ideas used in the proof of the
strong non-collapsing estimate played a crucial role.

The material presented in this section follows closely what Perelman gave in a private lecture at MIT in April, 2003. As such,  it naturally overlaps considerably 
with the earlier notes by Sesum-Tian \cite{Se2}  on Perelman's work\footnote{Perelman's private lecture was attended by a very small audience, including this author 
and the authors of \cite{Se2}.}.
I also presented Perelman's uniform estimates at the Geometry and Analysis seminar at Columbia University in fall 2005.

\begin{theorem} Let $X^n$ be a Fano manifold and $g_{i\bar j}(t)$, $0\le t<\infty$, be the solution to the NKRF
$$\frac{\partial}{\partial t} g_{i\bar j} = - R_{i\bar j} + g_{i\bar j}, \quad g(0) = \tilde g  \eqno(6.1)$$
with the initial metric $g_{0}=\tilde g$ satisfying $[\omega_0]=\pi c_1(X)$. Let $f=f(t)$ be the Ricci potential of $g_{i\bar j}(t)$ satisfying
$$ -R_{i\bar j}(t)+g_{i\bar j}(t)=\partial_i\partial_{\bar j} f \eqno(6.2) $$ and the normalization
$$\int_{X^n} e^{-f} dV=(2\pi)^n. \eqno(6.3)$$  Then there exists
a constant $C>0$ such that

\medskip

(i) $|R|\le C$ on $X^n \times [0,\infty)$;

\medskip

(ii) ${\rm diam} (X^n, g_{i\bar j}(t)) \le C$;

\medskip

(iii)  $||f||_{C^1}\le C$ on $X^n \times [0,\infty)$.

\end{theorem}

\begin{pf} First of all, by Lemma 2.5, we know that under (6.1)  the scalar curvature $R$ evolves according to the equation
$$ \frac{\partial} {\partial t} R=\Delta R +|Rc|^2-R.$$

\begin{lemma} There exists a constant $C_1>0$ such that
the scalar curvature $R$ of the NKRF (6.1) satisfies the estimate
$$R(x, t)\ge -C_1.$$
for all $t\ge 0$ and all $x\in X^n$.
\end{lemma}

\begin{pf} Let $R_{\min}(0)$ be the minimum of $R(x, 0)$ on $X^n$. If $R_{\min}(0)\ge 0$, then by the maximum principle, we have $R(x, t)\ge 0$
for all $t>0$ and all $x\in X^n$.

Now suppose $R_{\min}(0)<0$. Set $F(x, t)=R(x, t)-R_{\min}(0)$. Then, $F(x, 0)\ge 0$ and $F$ satisfies
$$ \frac{\partial} {\partial t} F=\Delta F +|Rc|^2-F -R_{\min}(0)>\Delta F +|Rc|^2-F.$$
Hence it follows again from the maximum principle that $F\ge 0$ on $X^n\times [0, \infty)$, i.e.,
$$R(x, t)\ge R_{\min}(0)$$
for all $t>0$ and all $x\in X^n$.
\end{pf}

Next, we consider the Ricci potential $f$ satisfying (6.2) and the normalization
(6.3).   Note that it follows from (6.2) that
$$n-R=\Delta f. \eqno(6.4)$$

Also, let $\varphi=\varphi(t)$ be the K\"ahler potential,
$$g_{i\bar j}(t)=\tilde g_{i\bar j}+ \partial_i\partial_{\bar j} \varphi, $$
so that $\varphi$ is a solution to the parabolic scalar equation
$${\varphi}_{t} = \log \frac{\det(\tilde g_{i\bar j}+\partial_i\partial_{\bar j}\varphi)} {\det(\tilde g_{i\bar j})}+\tilde f +\varphi +b(t),$$
where $b(t)$ is a function of $t$ only.

Since $\partial_i\partial_{\bar j} {\varphi}_{t}=-R_{i\bar j}+g_{i\bar j}$, by adding a function of $t$ only to $\varphi$ if necessary, we can assume
$$f=\varphi_{t}. \eqno(6.5)$$
Thus,  $f$ satisfies the parabolic equation
$$ f_{t}=\Delta f +f -a(t) \eqno(6.6)$$
for some function a(t) of $t$ only.

By differentiating the constraint (6.3), we get
$$\int_{X^n} e^{-f} (-f_{t} +n-R)dV=0.$$
Hence, by combining with (6.4) and (6.6), it follows that
$$a(t)= (2\pi)^{-n} \int_{X^n} f e^{-f} dV. \eqno(6.7)$$

\begin{lemma} There exists a constant $C_2>0$ such that, for all $t\ge 0$,
$$ -C_2\le \int_{X^n} f e^{-f} dV \le C_2.$$

\end{lemma}

\begin{pf}   The second inequality is easy to see. Now we prove the first inequality.
By Lemma 4.3 and (6.4), we have
\begin{align*}
A=:\! \mu (g_{i\bar j}(0), 1) & \le \mu (g_{i\bar j}(t), 1) \\
& \le (2\pi)^{-n}\! \int_X (R+|\nabla f|^2 +f-2n) e^{-f} dV\\
& =  (2\pi)^{-n}\! \int_X (-\Delta f + |\nabla f|^2 +f -n) e^{-f} dV\\
&= (2\pi)^{-n}\! \int_X (f -n) e^{-f} dV.
\end{align*}
Therefore,
$$(2\pi)^{-n}\! \int_{X^n} f e^{-f} dV\ge A+n.$$
\end{pf}

\begin{lemma} There exists a constant $C_3>0$ such that
$$ f \ge - C_3$$
for all $t\ge 0$ and all $x\in X^n$.
\end{lemma}

\begin{pf}  We argue by contradiction. Suppose the Ricci potential $f$ is very negative at some  time $t_0>0$ and some point $x_0\in X^n$  so that
$$f(x_0, t_0)<<-1.$$ Then, there exists  some open neighborhood $U\subset X^n$ of $x_0$ such that $$f(x, t_0)<<-1, \qquad \forall x\in U. \eqno(6.8)$$
On the other hand, by (6.4), (6.6), Lemma 6.1, (6.7), and Lemma 6.2, we have
$$ f_{t}=n-R +f -a(t)\le  f +C \eqno(6.9)$$ for some uniform constant $C>0$.  

Let us assume $ f(\cdot, t)$ and $\varphi (\cdot, t)$ achieve their maximum at $x_t$ and $x_t^{*}$ respectively.
From the constraint (6.3),
it is clear that for each $t>0$, we have a uniform lower estimate
$$f(x_t, t)=\max_{X} f(\cdot, t) \ge -C$$
for some $C>0$ independent of $t$. Moreover, it follows form (6.5) and (6.9) that
$$ (f-\varphi)_{t} \le C,$$
so
$$ f(\cdot, t)-\varphi (\cdot ,t)\le \max_{X} (f-\varphi)(\cdot, t_0) +Ct.$$
Therefore,
$$\varphi(x^{*}_t, t)\ge \varphi(x_t, t)\ge  f(x_t, t)-\max_{X} (f-\varphi)(\cdot, t_0)-Ct\ge -Ct, \quad \forall t>>t_0. \eqno(6.10)$$

On the other other, by (6.9), we have
$$f(x, t) \le e^{t-t_0}(C+f(x, t_0)) \eqno(6.11)$$
for $t\ge t_0$ and $x\in X^n$. In particular,  by (6.8), we have
$$f(x, t)\le -Ce^{-t_0} e^{t},  \qquad \forall t>t_0, \forall x\in U. \eqno(6.12)$$
Then (6.5) and (6.12) together imply that
$$\varphi(x, t)\le \varphi(x, t_0)-Ce^{-t_0} e^{t}  +C \le -C' e^{t}, \qquad \forall t>> t_0, \forall x\in U. \eqno(6.13)$$

Next, we claim (6.13) implies
$$\varphi (x_{t}^{*}, t)\le -Ce^{t}+ C' \eqno(6.14)$$ for some $C'>0$ independent of $t>>t_0$.  To see this, note that, with respect to the initial metric $g_0$,  we have
$$\varphi (x_{t}^{*}, t)=\frac{1} {V_0(X^n)}\int _X \varphi (\cdot, t) dV_0 - \frac{1} {V_0(X^n)}\int _X \Delta_0\varphi (\cdot, t)G_0(x_t^{*}, \cdot) dV_0,\eqno(6.15)$$
where $V_0(X^n)=\Vol (X^n, g_0)$ and $G_0(x_t^{*}, \cdot)$ denotes a positive Green's function with pole at $x_t^{*}$.

Since $n+\Delta_0\varphi=\tilde g^{i\bar j} g_{i\bar j}(t)>0,$ the second term on the RHS of (6.15) can be estimated by
$$ - \frac{1} {V_0(X^n)}\int _X \Delta_0\varphi (\cdot, t)G_0(x_t, \cdot) dV_0\le  \frac{n} {V_0(X^n)}\int _X G_0(x_t, \cdot) dV_0=:C''. \eqno(6.16)$$
On the other hand, by using (6.12), it follows that
$$\frac{1} {V_0(X^n)}\int _X \varphi (\cdot, t) dV_0 \le \frac{V_0(X\setminus U)} {V_0(X)} \varphi (x_{t}^{*}, t) - \frac{V_0(U)} {V_0(X)} Ce^{t}.\eqno(6.17)$$
Therefore, by (6.15)-(6.17), we have
$$\alpha \varphi (x_{t}^{*}, t)\le C''- \alpha Ce^{t} $$ for $\alpha =V_0(U)/V_0(X)>0$.   This proves (6.14),  a contradiction to (6.10).
\end{pf}

\begin{lemma} There exists constant $C_4>0$ such that, for all $t\ge 0$,

\medskip

(a) $|\nabla f|^{2}\le  C_4(f+2C_3)$;

\medskip

(b) $ R \le  C_4 (f+2C_3)$.

\end{lemma}

\begin{pf} This  is essentially a parabolic version of Yau's gradient estimate in \cite{Yau75} (see also  \cite{ScY}).

First of all, from $|\nabla f|^2= g^{i\bar j}\partial_i f\partial_{\bar j}f$, the NKRF, and (6.6), we obtain
\begin{align*}
\frac{\partial }{\partial t} |\nabla f|^2 & = (R_{i\bar j}- g_{i\bar j}) \partial_i f\partial_{\bar j}f + g^{i\bar j} (\partial_i f_t\partial_{\bar j}f + \partial_i f\partial_{\bar j}f_t)\\
&= g^{i\bar j} [\partial_i (\Delta f)\partial_{\bar j}f + \partial_i f\partial_{\bar j}(\Delta f)] + Rc(\nabla f, \nabla f) +|\nabla f|^2.
\end{align*}
On the other hand, the Bochner formula gives us
$$\Delta |\nabla f|^2= |\nabla \bar\nabla f|^2 +|\nabla\nabla f|^2 + g^{i\bar j} [\partial_i (\Delta f)\partial_{\bar j}f + \partial_i f\partial_{\bar j}(\Delta f)] + Rc(\nabla f, \nabla f).$$
Hence, we have
$$ \frac{\partial }{\partial t} |\nabla f|^2= \Delta |\nabla f|^2  - |\nabla \bar\nabla f|^2  - |\nabla\nabla f|^2 + |\nabla f|^2. \eqno (6.18) $$
Also, by (6.2), we have
$$ |Rc|^2 +n-2R= |\nabla \bar\nabla f|^2. \eqno (6.19) $$
Thus,  from the evolution equation on $R$, we have
$$ \frac{\partial} {\partial t} R\le \Delta R +|\nabla \bar\nabla f|^2+R$$
Therefore,  for any $\alpha \ge 0$, we obtain
$$ \frac{\partial }{\partial t} (|\nabla f|^2+\alpha R) \le \Delta (|\nabla f|^2+\alpha R)  - (1-\alpha) ( |\nabla \bar\nabla f|^2 + |\nabla\nabla f|^2) + (|\nabla f|^2 +\alpha R). \eqno (6.20) $$

Next, take $B=2C_3$ so we have $f+B>1$, and  set
$$u= \frac {|\nabla f|^2+\alpha R} {f+B}.\eqno (6.21)$$
Then, we have
$$u_t=\frac {(|\nabla f|^2 +\alpha R)_t} {f+B} - \frac{u} {(f+B)} f_t $$
and
$$\nabla u=  \frac {1} {f+B}\nabla (|\nabla f|^2+\alpha R) - \frac {|\nabla f|^2+\alpha R} {(f+B)^2}\nabla f. \eqno (6.22)$$
On the other hand,  since $ |\nabla f|^2+\alpha R=u (f+B)$, we have
$$\Delta (|\nabla f|^2+\alpha R)=(f+B)\Delta u +u\Delta f +\nabla u\cdot.\bar{\nabla} f +{\bar\nabla} u\cdot.\nabla f$$
or
$$ \Delta u =\frac {\Delta (|\nabla f|^2+\alpha R)} {f+B} -\frac {u \Delta f}{f+B} -\frac {\nabla u\cdot \bar{\nabla} f +{\bar\nabla} u\cdot \nabla f} {f+B}.$$
Therefore,
$$u_t \le \Delta u - (1-\alpha) \frac {(|\nabla \bar\nabla f|^2  + |\nabla\nabla f|^2)}{f+B} +\frac {\nabla u\cdot \bar{\nabla} f +{\bar\nabla} u\cdot \nabla f} {f+B}  +\frac{B+a(t)} {f+B} u. \eqno (6.23)$$
Notice, by (6.22),  we have
$$\nabla u\cdot {\bar\nabla} f=  \frac {1} {f+B} \nabla (|\nabla f|^2+\alpha R)\cdot {\bar\nabla} f - \frac {(|\nabla f|^2+\alpha R)|\nabla f|^2} {(f+B)^2}.\eqno(6.24)$$
Now the trick (see, e.g., p. 19 in \cite{ScY}) is to use (6.24) and express
$$\frac {\nabla u\cdot \bar{\nabla} f} {f+B}  =(1-2\epsilon) \frac {\nabla u\cdot \bar{\nabla} f} {f+B}+\frac {2\epsilon} {f+B} \( \frac {\nabla (|\nabla f|^2+\alpha R)\cdot {\bar\nabla} f} {f+B}
- \frac {|\nabla f|^2(|\nabla f|^2+\alpha R)} {(f+B)^2}\).\eqno (6.25)$$

We are ready to conclude the proof of Lemma 6.4.

\medskip
{\bf Part (a)}: Take $\alpha=0$ so that $u=|\nabla f|^2/(f+B)$. By plugging (6.25) into (6.23),   we get
\begin{align*}
u_t  & \le \Delta u - (1-4\epsilon) \frac {|\nabla \bar\nabla f|^2  + |\nabla\nabla f|^2}{f+B} + (1-2\epsilon) \frac {\nabla u\cdot \bar{\nabla} f +{\bar\nabla} u\cdot \nabla f} {f+B}\\
 & \ \ \ -\frac{\epsilon} {f+B}\( |2\nabla \bar\nabla f-\frac{\nabla f\bar\nabla f}{f+B}|^2  + |2\nabla\nabla f -\frac{\nabla f\nabla f}{f+B}|^2\)\\
    & \ \ \ + \frac{1} {(f+B)} \(-2\epsilon u^2 +(B+a) u\).\\
\end{align*}
For any $T>0$, suppose $u$ attains its maximum at $(x_0, t_0)$ on $X^n\times [0, T]$,
then  $$ u_t(x_0, t_0)\ge 0, \quad \nabla u(x_0, t_0)=0, \quad {\mbox and} \quad  \Delta u(x_0, t_0)\le 0. \eqno(6.26)$$ Thus, by choosing $\epsilon=1/8$, we arrive at
$$u(x_0, t_0)\le  4(B+a).$$
Therefore, since $T>0$ is arbitrary, we have shown that
$$\frac{|\nabla f|^2}{f+B} \le 8C_3+4C_2 \eqno(6.27)$$
on $X^n\times [0, \infty)$.

\medskip
{\bf Part (b)}:  Choose $\alpha=1/2$ so that
$$ u=  \frac {|\nabla f|^2+R/2} {f+B}.$$
Then, from (6.23) and (6.19),  we obtain
$$u_t \le \Delta u -  \frac{1}{2}\frac {|Rc|^2-2R}{f+B} +\frac {\nabla u\cdot \bar{\nabla} f +{\bar\nabla} u\cdot \nabla f} {f+B}  +\frac{B+a} {f+B} u.$$
Again, for any $T>0$, suppose $u$ attains its maximum at $(x_0, t_0)$ on $X^n\times [0, T]$. Then (6.26) holds, and  hence
$$0\le -\frac{1}{2n}\(\frac{R}{f+B}\)^2\!(x_0, t_0) + \frac{R}{f+B} (x_0, t_0)\(1+\frac{B+a}{2(f+B)}\) + (8C_3+4C_2)(B+a).$$ Here we have used the fact that $|Rc|^2\ge R^2/n$, $2f+B\ge 0$, $f+B\ge 1$,
and (6.27).  It then follows easily that $\frac{R}{f+B}(x_0,t_0)$ is bounded from above uniformly. Therefore, by Part (a), $\frac{R}{f+B}(x,t)$ is bounded uniformly on $X^n\times [0,T]$
for arbitrary $T>0$.
\end{pf}

\medskip
Clearly, Lemma 6.4 (a) implies that $\sqrt{f+2C_3}$ is Lipschitz.  From now on we assume the Ricci potential $f(\cdot, t)$ attains its minimum at a point $\hat x \in X^n$, i.e., $f(\hat x, t)=\min_X f(\cdot, t)$.
Then, by (6.3), we know $$ f(\hat x, t)\le C$$ for some $C>0$ independent of $t$. \\

\begin{corollary}  There exists a constant $C>0$ such that $\forall t>0$ and $ \forall x\in X$,

\medskip

(i) $ f(x, t)\le C[1+d^2_t(\hat x, x)]$;

\medskip

(ii) $|\nabla f|^2 (x, t) \le C[1+d^2_t(\hat x, x)]$;

\medskip
(iii)
$R(x, t)\le C[1+d^2_t(\hat x, x)].$
\end{corollary}

\begin{pf} Set $h=f+2C_2>0$. Then, from Lemma 6.4 (i), we see that $\sqrt {h}$ is a Lipschitz function
satisfying $$|\nabla \sqrt{h}|^2\le C_4.$$

Hence, $\forall x\in X^n$,
$$|\sqrt {h}(x, t) -\sqrt {h}(\hat x, t)| \le C d_t(\hat x, x), $$ or
$$ \sqrt {h}(x, t)\le \sqrt {h}(\hat x, t) + C d_t(\hat x, x).$$
Thus, we obtain a uniform upper bound
$$f(x, t)\le h(x, t)\le C (d^2_t(\hat x, y) +1)$$ for some constant $C>0$ independent of $t$.
Now  (ii) and (iii) follow immediately from (i) and Lemma 6.4.
\end{pf}

\medskip

By Lemma 6.1 and Corollary 6.5, it remains to prove the following uniform diameter bound. \\

\begin{lemma} There exists a constant $C_5>0$ such that
$${\rm diam}_t (X)\!=: \! {\rm diam} (X^n, g_{i\bar j}(t))\le C_5$$ for all $t\ge 0.$
\end{lemma}

 \medskip
{\bf {\it Proof.}} \ For each $t>0$, denote by $A_t(k_1, k_2)$ the
annulus region defined by
$$A_t(k_1, k_2)=\{z\in X: 2^{k_1}\le d_{t}(x, \hat x)\le 2^{k_2}\},\eqno(6.28)$$
and by $$V_t(k_1, k_2)=\Vol (A_t(k_1, k_2))\eqno(6.29) $$ with respect to $g_{i\bar j}(t)$.

Note that each annulus $A_t(k, k+1)$ contains at least $2^{2k}$ balls $B_r$ of
radius $r=2^{-k}$. Also, for each point $x\in A_t(k, k+1)$, Corollary 6.1 (iii) implies that the scalar curvature is
bounded above by $R\le C2^{2k}$ on $B_{t}(x, r)$ for some uniform constant $C>0$.
Thus each of these balls $B_r$ has $\Vol(B_r)\ge \kappa (2^{-k})^{2n}$ by Theorem 4.2, so
we have
$$ V_t(k, k+1)\ge \kappa 2^{2k-1}2^{-kn}. \eqno(6. 30)$$

\medskip
{\bf Claim 6.1}: For each small $\epsilon >0$, there exists a large constant $D=D(\epsilon) >0$ such that if ${\rm diam}_t (X) >D$, then one can find large positive constants
$k_2> k_1>0$ with the following properties:

$$V_t(k_1, k_2)\le \epsilon \eqno(6. 31)$$ and
$$V_t(k_1, k_2)\le 2^{10n}V_t(k_1+2, k_2-2).\eqno(6. 32)$$

\medskip
{\sl Proof}. (a) follows from the fact that $V_{t}(X^n)=V_0(X^n)$ and the assumption ${\rm diam}_t (X) >>1$.

Now suppose (a) holds but  not (b), i.e.,
$$V_t(k_1, k_2)> 2^{10n}V_t(k_1+2, k_2-2).$$
Then we consider whether or not
$$V_t(k_1+2, k_2-2)\le 2^{10n}V_t(k_1+4, k_2-4).$$
If yes, then we are done. Otherwise we repeat the process.

After $j$ steps, we either have
$$V_t(k_1+2(j-1), k_2-2(j-1))\le 2^{10nj}V_t(k_1+2j, k_2-2j),\eqno(6.33)$$
or
$$V_t(k_1, k_2)> 2^{10nj}V_t(k_1+2j, k_2-2j).\eqno(6.34)$$
Without loss of generality, we may assume $k_1+2j \approx k_2-2j$ by choosing a large number $K>0$ and pick
$k_1\approx K/2, k_2\approx 3K/2$.
Then, when $j\approx K/4$ and using (6.30), this implies that
$$\epsilon\ge V_t(k_1, k_2)\ge 2^{10nK/4}V_t(K, K+1)\ge \kappa 2^{2K(n/4-1)}.$$
So after some finitely many steps $j\approx K(\epsilon)/4$, (6.33) must hold. Therefore, we have found $k_1$ and $k_2\approx 3k_1$ satisfying
both (6.31) and (6.32).

\medskip

{\bf Claim 6.2}: There exist constants $r_1>0$ and $r_2>0$, with $r_1\in [2^{k_1}, 2^{k_1+1}]$ and $r_2\in
[2^{k_2}, 2^{k_2+1}]$, such that
$$\int_{A_t(r_1, r_2)} RdV_{t} \le CV_t(k_1, k_2). \eqno (6.35)$$

{\sl Proof}. First of all,  since
$$ \frac{d}{dr} \Vol(B(r)) =\Vol (S(r),$$ we have
$$V(k_1, k_1+1)=\int_{2^{k_1}}^{2^{k_1+1}} \Vol(S(r)) dr.$$
Here $S_{r}$ denotes the geodesic sphere of radius $r$ centered at
$\hat x$ with respect to $g_{i\bar j}(t)$. Hence, we can choose $r_1\in [2^{k_1}, 2^{k_1+1}]$
such that
$$\Vol(S_{r_1})\le \frac{V_t(k_1, k_2)}{2^{k_1}}, $$
for otherwise
$$V(k_1, k_1+1)> \frac{V_t(k_1, k_2)} {2^{k_1}}2^{k_1}=V_t(k_1, k_2),$$
a contradiction because $k_2>k_1+1$.
Similarly, there exists  $r_2\in [2^{k_2-1}, 2^{k_2}]$ such that
$$\Vol(S_{r_2})\le \frac{V_t(k_1, k_2)}{2^{k_2}}.$$

Next, by integration by parts and Corollary 6.1(ii),
\begin{align*} |\int_{A_t(r_1, r_2)}\Delta f|& \le
\int_{S_{r_1}}
|\nabla f| +\int_{S_{r_2}} |\nabla f|\\
& \le \frac{V_t(k_1, k_2)}{2^{k_1}}C2^{k_1+1} + \frac{V_t(k_1,
k_2)}{2^{k_2}}C2^{k_2+1}\\
& \le C V_t(k_1, k_2).
\end{align*}
Therefore, since $R+\Delta f=n$, it follows that
$$\int_{A_t(r_1, r_2)} R dV_t\le CV_t(k_1, k_2), $$ proving Claim 6.2. \\

Now we argue by contradiction to finish the proof:  Suppose ${\rm diam}_{t} (X^n)$ is unbounded  for $0\le t<\infty$. Then, for any sequence $\epsilon_{i}\to 0$, there exists
a time sequence $\{t_{i}\}\to \infty$ and  $k_2^{(i)}>k_1^{(i)}>0$ for which Claim 6.1 holds.  Pick smooth cut-off functions $0\le\zeta_i(s) \leq 1$ defined
on $\mathbb R$ such that
$$ {\zeta_i (s) =\left\{
       \begin{array}{lll}
  1, \ \ \qquad  2^{k_1^{(i)}+2}\le s\le 2^{k_2^{(i)}-2},\\[4mm]
  0, \ \ \qquad \mbox{outside}\ [r_1^{(i)}, r_2^{(i)}],
       \end{array}
    \right.}$$
and $|\zeta'|\leq 1$ everywhere. Here $r_1^{(i)}\in [2^{k_1^{(i)}}, 2^{k_1^{(i)}+1}]$ and $r_2^{(i)} \in [2^{k_2^{(i)}-1}, 2^{k_2^{(i)}}]$ are chosen as in Claim 6.2.
Define
$$u_i=e^{L_i}\zeta_i({d_{t_i}(x,\hat x_i)}),
$$
where $f(\hat x_i, t_i)=\min_X f(\cdot, t_i) $ and the constant $L_i$ is chosen so that $$(2\pi)^{n}=\int_X
u_i^2dV_{t_i}=e^{2L_i}\int_{A(r_1^{(i)}, r_2^{(i)})}\zeta_i^2 dV_{t_i}.\eqno(6.36)$$
Note that by Claim 1, $V_{t_i} (k_1^{(i)}, k_2^{(i)})\le \epsilon_i\to 0$.  Hence (6.36) implies $L_i\to \infty$.

Now, by Lemma 4.3 and similar to the proof of Theorem 4.1,  we have
\begin{align*}
\mu (g (0), 1) & \le \mu (g (t_i), 1) \\
& \le (2\pi)^{-n}\int_X (Ru_i^2+4|\nabla u_i|^2-u_i^2\log
u_i^2-2nu_i^2) dV_{t_i}\\
& = (2\pi)^{-n}e^{2L_i}\int_{A_{t_i}(r_1^{(i)}, r_2^{(i)})} (R \zeta_i^2
+4|\zeta_i'|^2-\zeta_i^2\log\zeta^2_i-2L_i\zeta_i^2-2n\zeta_i^2) dV_{t_i}\\
& = -2(L_i+n) + (2\pi)^{-n}e^{2L_i}\int_{A_{t_i}(r_1^{(i)}, r_2^{(i)})} (R
\zeta_i^2 + 4|\zeta_i'|^2-\zeta_i^2\log\zeta^2_i) dV_{t_i}.\\
\end{align*}

Now, by Claim 6.2 and Claim 6.1,  we have
\begin{align*}
e^{2L_i}\int_{A_{t_i}(r_1^{(i)}, r_2^{(i)})} R \zeta_i^2 dV_{t_i} & \le
Ce^{2L_i}V_{t_i}(k_1^{(i)}, k_2^{(i)}) \\
& \le Ce^{2L_i}2^{10n}V_{t_i}(k_1^{(i)}+2, k_2^{(i)}-2)\\
& \le C2^{10n}\int_{A_{t_i}(r_1^{(i)}, r_2^{(i)})} u_i^2 dV_{t_i}\le
C2^{10n}(2\pi)^{n}.
\end{align*}
On the other hand, using $|\zeta'_i|\le 1$ and $-s\log s\le e^{-1}$
for $0\le s\le 1$, we also have
\begin{align*}
e^{2L_i}\int_{A_{t_i}(r_1^{(i)}, r_2^{(i)})} (4|\zeta_i'|^2-2\zeta_i^2\log\zeta_i)dV_{t_i}  & \le
Ce^{2L_i}V_{t_i}(k_1^{(i)}, k_2^{(i)})\\
&  \le C2^{10n}(2\pi)^{n}.
\end{align*}
Therefore,
$$\mu(g(0), 1)\le -2(L_i +n)+C$$ for some uniform constant $C>0$.
But this is a contradiction to  $\{L_i\} \to \infty$.
\end{pf}

\section {Remarks on the formation of singularities in KRF}

\medskip

Consider a solution $g_{ij}(t)$ to the Ricci flow
$$\frac{\partial g_{ij}(t)}{\partial t}=-2R_{ij}(t) $$
on $M\times
[0,T)$, $T\le +\infty$, where either $M$ is compact or at each
time $t$ the metric $g_{ij}(t)$ is complete and has bounded curvature. We say
that $g_{ij}(t)$ is a {\it maximal} solution of the Ricci flow if
either $T=+\infty$ or $T<+\infty$ and the norm of its curvature
tensor $|Rm|$ is unbounded as $t\to T$. In the latter case, we say
$g_{ij}(t)$ is a {\sl singular} solution to the Ricci flow with singular time $T$.
We emphasize that by singular solution $g_{ij}(t)$ we mean the curvature of $g_{ij}(t)$ is not uniformly bounded on
$M^n\times [0,T)$, while $M^n$ is a smooth manifold and $g_{ij}(t)$ is a smooth complete metric for each $t<T$.

As in the minimal surface theory and harmonic map theory, one usually tries to understand the structure of a singularity  by
rescaling the solution (or blow up) to obtain a sequence of solutions and study its limit. For the Ricci flow, the theory was
first developed by Hamilton in \cite{Ha95F} and further improved by Perelman \cite{P1, P2}.

Now we apply Hamilton's theory to investigate singularity formations of KRF (2.1) on compact Fano manifolds.
Consider a (maximal) solution $\hat g_{i\bar j}(s)$ to KRF (2.1) on $X^n\times [0,1)$ and the corresponding
solution $g_{i\bar j}(t)$ to NKRF (2.5) on $X^n\times [0,\infty)$, and let us denote by
$$ \hat K_{\max}(s)=\max_{X}|\hat{Rm}(\cdot, s)|_{\tilde g (s)} \quad {\mbox{and}} \qquad K_{\max}(t)=\max_{X}|{Rm}(\cdot, t)|_{g (t)}.
$$
According to Hamilton \cite{Ha95F}, one can classify  maximal
solutions to KRF (2.1) on any compact Fano manifold $X^n$ into Type I and Type II:

\medskip
$\mbox{\textbf{Type I:}  } \ \ \ \! \ \ \ \ \lim \sup_{s\to 1}(1-s){\hat K}_{\max}(s)<+\infty;\ $

\smallskip $\mbox{\textbf{Type II:}} \ \ \ \ \  \ \lim \sup_{s\to 1} (1-s){\hat K}_{\max}(s)=+\infty. $

\medskip
\noindent On the other hand, by Corollary 2.1, $ \hat K_{\max}(s)$ and $K_{\max}(t)$ are related by
$$ (1-s) \hat K_{\max}(s)=K_{\max}(t(s)).$$ Thus, we immediately get

\begin{lemma} Let $\hat g_{i\bar j}(s)$ be a solution to KRF (2.1) on $X^n\times [0,1)$ and $g_{i\bar j}(t)$  be the corresponding
solution to NKRF (2.5) on $X^n\times [0,\infty)$. Then,

\begin{itemize}
\smallskip

\item[(a)]  $\hat g_{i\bar j}(s)$ is a Type I solution if and only if $g_{i\bar j}(t)$ is a nonsingular
solution, i.e., $K_{\max}(t)\le C$ for some constant $C>0$ for all $t\in [0, \infty)$;

\smallskip
\item[(b)] $\hat g_{i\bar j}(s)$ is a Type II solution if and only if $g_{i\bar j}(t)$ is a singular solution. 
\end{itemize}
\end{lemma}

For each type of (maximal) solutions $\hat g_{i\bar j}(s)$ to KRF (2.1) or the corresponding solutions $g_{i\bar j}(t)$ for NKRF (2.5),
following Hamilton \cite{Ha95F} (see also Chapter 4 of \cite{CaoZhu06}) we define a corresponding type of limiting singularity models.

\begin{definition} A solution $g^{\infty}_{i\bar j}(t)$ to KRF on a complex
manifold $X^n_{\infty}$ with complex structure $J_{\infty}$, where either $X^n_{\infty}$ is compact or at each time $t$ the
K\"ahler metric $g^{\infty}_{i\bar j}(t)$ is complete and has bounded curvature, is
called a Type I or Type II {\bf singularity model} if it is not flat and of one of
the following two types:

\medskip
\noindent \textbf{Type I}: $g^{\infty}_{i\bar j}(t)$ exists for
$t\in(-\infty,\Omega)$ for some $\Omega$ with
$0<\Omega<+\infty$ and
$$|Rm^{\infty}|(x, t) \leq\Omega/(\Omega-t)
$$
\hskip 1.6cm everywhere on $X^n_{\infty}\times (-\infty,\Omega)$ with equality somewhere at $t=0$;

\vskip 0.1cm\noindent \textbf{Type II}: $g^{\infty}_{i\bar j}(t)$ exists for
$t\in(-\infty,+\infty)$  and
$$|Rm^{\infty}|(x, t)\leq 1 \ \ \ \ \ \ \ \ \
$$
\hskip 1.6cm everywhere on $X^n_{\infty}\times (-\infty,\Omega)$  with equality somewhere at $t=0$.
\end{definition}

With the help of the strong $\kappa$-noncollapsing theorem, we  can apply Hamilton's Type I and Type II blow up arguments to get the following result, 
a K\"ahler analog of  Theorem 16.2 in \cite{Ha95F}: 

\begin{theorem}  For any (maximal) solution $\hat g_{i\bar j} (s)$, $0\le s<1$, to KRF (2.1) on compact Fano manifold $X^n$ (or  the corresponding solution 
$g_{i\bar j} (t)$ to NKRF (2.5) on $X^n\times [0, \infty)$), which is of either Type I or Type II, there exists a sequence of dilations of the solution which
converges in $C_{loc}^{\infty}$ topology to a singularity model $(X^n_{\infty}, J_{\infty}, g^{\infty}(t))$ of the corresponding Type. Moreover, the Type I singularity model
$(X^n_{\infty}, J_{\infty}, g^{\infty}(t))$ is  compact  with $X^n_{\infty}=X^n$ as a smooth manifold, while the Type II singularity model $(X^n_{\infty}, J_{\infty}, g^{\infty}(t))$ is complete 
noncompact.  
\end{theorem}

\begin{pf}  {\bf Type I case:}  \  Let
$$
\Omega=: \limsup_{t\rightarrow 1}
(1-s)\hat K_{\max}(s)<+\infty.
$$
First we note that $\Omega>0$. Indeed by the evolution equation of
curvature,
$$
\frac{d}{ds}{\hat K_{\max}(s)}\leq{\rm Const}\, \cdot \hat K_{\max}^2(s).
$$
This implies that
$$
\hat K_{\max}(s)\cdot (1-s)\geq {\rm Const}\, >0,
$$
because
$$
\limsup_{t\rightarrow 1}\hat K_{\max}(s)=+\infty.
$$
Thus $\Omega$ must be positive.

Next we choose a sequence of points $x_k$ and times $s_k$ such that
$s_k\rightarrow 1$ and
$$
\lim_{k\rightarrow \infty}(1-s_k)|\hat Rm|(x_k,s_k)=\Omega.
$$
Denote by
$$
Q_k=|\hat Rm|(x_k,s_k).
$$
Now translate the time so that $s_k$ becomes 0 in the new time, and dilate in space-time by
the factor $Q_k$ (time like distance squared) to get the rescaled solution
$$
\hat {g}_{i\bar j}^{(k)}(\hat {t})
=Q_{k}  \ \!\hat g_{i\bar j}(s_k+Q_k^{-1}\hat {t})
$$
to the KRF
$$
\frac{\partial}{\partial
\hat {t}}\hat {g}_{i\bar j}^{k} = -2\hat {R}_{i\bar j}^{(k)},
$$
where $\hat {R}_{i\bar j}^{(k)}$ is the Ricci tensor of
$\hat {g}_{i\bar j}^{(k)}$, on the time interval $[-Q_{k} s_k, Q_{k}(1-s_k))$, with
$$
Q_{k} s_k=s_k |\hat Rm|(x_k,s_k)\to \infty \quad \mbox{and} \quad  Q_{k} (1-s_k)=(1-s_k)|\hat Rm|(x_k,s_k)\to \Omega.
$$

For any $\epsilon>0$ we can find a time $\tau <1 $ such that for
$s\in [\tau,1)$,
$$
|\hat Rm|\leq (\Omega+\epsilon)/(1-s)
$$
by the assumption. Then for $\hat {t}\in
[Q_{k} (\tau-s_k), Q_k (1-s_k))$, the curvature of
$\hat {g}_{i\bar j}^{(k)}(\hat {t})$ is bounded by
\begin{align*}
|\hat {Rm}^{(k)}| & = Q_{k}^{-1} |\hat Rm (\hat g)| \\
& \leq \frac {\Omega+\epsilon} {Q_k (1-s)} =\frac {\Omega+\epsilon} {Q_k (1-s_k)+ Q_k (s_k-s)}\\
&  \rightarrow (\Omega+\epsilon)/(\Omega-\hat {t}), \qquad
\text{as } \; k\rightarrow +\infty.
\end{align*}
With the above curvature bound and the injectivity radius estimates coming from $\kappa$-noncollapsing, one can apply 
Hamilton's compactness theorem (cf \cite{Ha95F} or Theorem 4.1.5 in \cite{CaoZhu06}) to get a subsequence
of $\hat {g}_{i\bar j}^{(k)}(\hat {t})$ which converges
in the $C_{loc}^{\infty}$ topology to a limit metric
$g_{i\bar j}^{(\infty)}(t)$ in the Cheeger sense on 
$(X^n, J_{\infty})$  for some complex structure $J_{\infty}$ such that
$g_{i\bar j}^{(\infty)}(t)$ is a solution to
the KRF with $t \in(-\infty,\Omega)$ and its curvature satisfies the bound
$$
|Rm^{(\infty)}|\leq \Omega/(\Omega-t)
$$
everywhere on $X^n_{\infty} \times (-\infty,\Omega)$ with the equality
somewhere at $t=0$.

\medskip
{\bf Type II:}\index{type II!(a)} \  Take  a sequence  $S_k \rightarrow 1$ and pick space-time points $(x_k, s_k)$
such that, as $k\rightarrow +\infty$,
$$
Q_k (S_k-s_k)=  \max_{x\in X, s\leq S_k}(S_k-s)|\hat Rm|(x,s)
\rightarrow +\infty,
$$
where again we denote by $Q_k=|\hat Rm|(x_k,s_k)$. Now translate the time and dilate the solution as before to get
$$
\hat {g}_{i\bar j}^{(k)}(\hat {t})
=Q_k \hat g_{i\bar j}(s_k+ Q_k^{-1}\hat {t}),
$$
which is a solution to the KRF and satisfies the
curvature bound
\begin{align*}
|\hat{Rm}^{(k)}| &  =Q_k^{-1} |\hat Rm(\hat g)| \leq  \frac{(S_k-s_k)}{(S_k-s)}\\
& =\frac{Q_k (S_k-s_k)}{Q_k (S_k-s_k)-\hat{t}}\quad \mbox{ for }
\; \hat{t}\in\ [-Q_k s_k, Q_k (S_k-s_k)).
\end{align*}
Then as before, by
applying Hamilton's compactness theorem, there exists a
subsequence of $\hat{g}_{i\bar j}^{(k)}(\hat {t})$ which
converges in the $C_{loc}^{\infty}$ topology to a limit metric
$g_{i\bar j}^{(\infty)}(t)$ in the Cheeger sense on a limiting complex manifold
$(X^n_{\infty}, J_{\infty})$  such that
$g_{i\bar j}^{(\infty)}(t)$ is a complete solution to
the KRF  with $t \in(-\infty,+\infty)$,  and its curvature satisfies
$$
|Rm^{(\infty)}|\leq 1
$$
everywhere on $X^n_{\infty} \times (-\infty,+\infty)$ and the equality
holds somewhere at $t=0$.
\end{pf}

\medskip
\begin{remark} The injectivity radius bound needed in Hamilton's compactness theorem is satisfied due to the ``Little Loop Lemma" (cf. Theorem 4.2.4 in\cite{CaoZhu06}),
which is a consequence of Perelman's  $\kappa$-noncollapsing theorem.
\end{remark}

\medskip

Thanks to Perelman's monotonicity of $\mu$-entropy and the uniform scalar curvature bound in Theorem 6.1, we can say more about the singularity models
in Theorem 7.1.

First of all,  the following result on Type I singularity
models of KRF (2.1) is well-known (cf. \cite{Se}).

\medskip
\begin{theorem} \ Let  $\tilde g_{i\bar j}(s)$ be a Type I solution to KRF (2.1) on $X^n\times [0,1)$ and
$g_{i\bar j}(t)$  be the corresponding nonsingular solution to NKRF (2.5) on $X^n\times [0,\infty)$. Then there exists a sequence $\{t_k\}\to \infty$ such that $g^{(k)}_{i\bar j}(t) =:\! g_{i\bar j}(t+t_k)$
converges in the Cheeger sense to a gradient shrinking K\"ahler-Ricci soliton $g^{\infty} (t)$ on $(X^n, J_{\infty})$, where $J_{\infty}$ is a  certain
complex structure on $X^n$, possibly different from $J$.
\end{theorem}

\begin{pf} This is a consequence of Theorem 7.1, and the fact  that every compact Type I singularity model
is necessarily a shrinking gradient Ricci soliton (see \cite{Se}, \cite {Se2} or p.662 of \cite{PSSW08}; also Corollary 1.2 in \cite{CCZ03}).
\end{pf}

Next, for Type II solutions to the KRF, we have the following two
results. These results were known to R. Hamilton and the author \cite{CH04}
back in 2004\footnote{Theorem 7.3 and Theorem 7.4 were observed
by Hamilton and the author during the IPAM conference ``Workshop
on Geometric Flows: Theory and Computation" in February, 2004.},
and also observed independently by Ruan-Zhang-Zhang
\cite{RZZ} (see also \cite{ChWa2012}).
\medskip
\begin{theorem} 
\ Let 
$g_{i\bar j}(t)$  be a  singular solution to NKRF (2.5) on $X^n\times [0,\infty)$. Then there exists a sequence $\{t_k\}\to \infty$ and rescaled solution metrics $g^{(k)} (t)$
to KRF such that $(X^n, J, g^{(k)} (t))$ converges in the Cheeger sense to some noncompact limit $(X_{\infty}^n, J_{\infty}, g_{\infty} (t))$,
$-\infty<t<\infty$,  with the following properties:

\begin{itemize}

\item[(i)] $g_{\infty} (t)$  is Calabi-Yau (i.e, Ricci flat K\"ahler);

\item[(ii)] $|Rm|_{g_{\infty} (t)} (x, t) \le 1$ everywhere and with equality somewhere at $t=0$;

\item[(iii)] $(X_{\infty}^n, g_{\infty} (t))$ has maximal volume growth: for any $x_0\in X_{\infty}^n$ there exists a positive constant $c>0$ such that
$$ \Vol (B(x_0, r))\ge c r^{2n}, \qquad {\mbox{for all}} \  r>0.$$
\end{itemize}
\end{theorem}

\begin{pf} This is an immediate consequence of Theorem 7.1 and Theorem 6.1 (i). Indeed,
Theorem 7.1 implies the existence of a noncompact Type II singularity model
$(X_{\infty}^n, J_{\infty}, g_{\infty} (t))$ satisfying property
(ii).  Property (iii) follows from the
fact that the $\kappa$-noncollapsing property for KRF or NKRF in
Theorem 4.2 is dilation invariant, hence (4.17) and (4.18) holds
for each rescaled solution on larger and larger scales for the
same $\kappa>0$, hence the maximal volume growth in the limit of
dilations.  Finally, for property (i), note that the scalar
curvature $R$ of $g_{i\bar j}(t)$ is uniformly bounded on $X\times [0,\infty)$ by Theorem 6.1 and the rescaling factors
go to infinite, so we have $R^{\infty}=0$ everywhere in the limit of dilations. On the other hand,  since $g_{i\bar j}^{\infty} (t)$ is a solution to 
KRF, $R^{\infty}$ satisfies the evolution equation
$$ \frac{\partial} {\partial t} R^{\infty}=\Delta R^{\infty} +|Rc^{\infty}|^2.$$
Thus, we have $|Rc^{\infty}|^2=0$ everywhere hence $g_{\infty}$ is Ricci-flat.
\end{pf}

\medskip
\begin{theorem} 
\ Let $X^2$ be a Del Pezzo surface (i.e., a Fano surface) and let 
$g_{i\bar j}(t)$  be a singular solution to NKRF (2.5) on $X^2\times [0,\infty)$. Then the Type II limit space $(X_{\infty}^2, J_{\infty}, g_{\infty}) $
in Theorem 7.3 is a non-compact Calabi-Yau space satisfies the following properties:

\smallskip
\begin{itemize}

\smallskip
\item[(a)] $|Rm|_{g_{\infty}} \le 1$ everywhere on $X^2_{\infty}$ and with equality somewhere;

   \smallskip
\item[(b)] $(X_{\infty}^2, g_{\infty})$ has maximal volume growth: for any $x_0\in X_{\infty}^2$
there exists a positive constant $c>0$ such that $$ \Vol (B(x_0, r))\ge c r^{4}, \qquad {\mbox{for all}} \  r>0;$$

\smallskip
\item[(c)] $\int_{X^2_{\infty}} |Rm(g_{\infty})|^2 dV_{\infty} <\infty$.

\end{itemize}
\end{theorem}

\begin{pf} Clearly, we only need to verify property (c). But this follows from the facts the integral $$\int_{X^2} |Rm|^2 (x,t) dV_{t} $$ is dilation invariant
in complex dimension $n=2$ (real dimension 4); that it differs
from $\int_X R^{2} dV_{t}$ up to a constant depending only on the
K\"ahler class of $g(0)$ and the Chern classes $c_1(X)$ and
$c_2(X)$ (cf. Proposition 1.1 in \cite{Calabi});  and that, before
the dilations, $\int_X R^2 dV_{t}$ is uniformly bounded for all
$t\in [0, \infty)$ by the uniform scalar curvature bound in Theorem 6.1 (i).
\end{pf}

\smallskip
\begin{remark} The work of Bando-Kasue-Nakajima \cite{BKN} implies that the limiting Calabi-Yau surfaces in Theorem 7.4 are asymptotically locally Euclidean (ALE) of order 
at least 4.
\end{remark}

\smallskip
\begin{remark} Kronheimer \cite{Kr89} has classified ALE Hyper-K\"ahler surfaces (i.e., simply connected ALE  Calabi-Yau surfaces).
\end{remark}

\end{document}